\documentclass[11pt,a4paper]{amsart}
\usepackage[top=1in, bottom=1in, left=1in, right=1in]{geometry}
\usepackage{mathrsfs}
\usepackage{amsfonts}
\usepackage{amsmath}
\usepackage{amssymb}
\usepackage{dsfont}
\usepackage{physics}
\numberwithin{equation}{section}
\usepackage{graphicx}
\usepackage{times}
\usepackage{colonequals}
\usepackage[usenames,dvipsnames]{color}
\usepackage{comment}
\usepackage{mathtools}
\usepackage{bbm}
\usepackage{bm}
\usepackage{esvect}
\usepackage{cite}
\usepackage{hyperref}
\hypersetup{
colorlinks = true,
linkcolor={black},
urlcolor={blue},
citecolor={blue},    
urlcolor = {blue},
citebordercolor = {0.33 .58 0.33},
 linkbordercolor = {0.99 .28 0.23},
 breaklinks=true}

\renewcommand{\P}{\mathbb{P}}

\makeatletter
\renewcommand{\pmod}[1]{\allowbreak\mkern7mu({\operator@font mod}\,\,#1)}
\makeatother

\renewcommand{\O}{\mathcal{O}}
\newcommand{\PR}{\mathbb{P}} 
\newcommand{\kommentar}[1]{}

\newcommand{\N}{\mathbb{N}}
\newcommand{\Z}{\mathbb{Z}}

\newtheorem{thm}{Theorem}[section]
\newtheorem{defn}[thm]{Definition}
\newtheorem{prop}[thm]{Proposition}
\newtheorem{coro}[thm]{Corollary}
\newtheorem{lem}[thm]{Lemma}
\newtheorem{conj}[thm]{Conjecture}

\newcommand{\e}{\mathrm{e}}

\theoremstyle{remark}
\newtheorem{rem}[thm]{Remark}

\newcommand{\ds}{\displaystyle}
\newcommand{\pfrac}[2]{\left(\frac{#1}{#2}\right)} 
\newcommand{\be}{\begin{equation}}
\newcommand{\ee}{\end{equation}}

\newcommand{\flr}[1]{\lfloor #1 \rfloor}

\newcommand{\alg}[1]{\begin{align*}{#1}\end{align*}}
\newcommand{\1}[1]{\label{#1}}

\newcommand{\E}{\mathbb{E}}
\newcommand*{\defeq}{\mathrel{\vcenter{\baselineskip0.5ex 
	\lineskiplimit0pt\hbox{\scriptsize.}
	\hbox{\scriptsize.}}}=}

\def\cH{\mathcal{H}}

\newcommand{\edit}[1]{{#1}}

\renewcommand{\v}{\varepsilon}
\renewcommand{\a}{\mathbf{a}}
\renewcommand{\le}{\leqslant}
\renewcommand{\leq}{\leqslant}
\renewcommand{\ge}{\geqslant}
\renewcommand{\geq}{\geqslant}

\DeclareFontFamily{OT1}{rsfs}{}
\DeclareFontShape{OT1}{rsfs}{n}{it}{<-> rsfs10}{}
\DeclareMathAlphabet{\mathscr}{OT1}{rsfs}{n}{it}

\newcommand{\cA}{\mathcal{A}}
\newcommand{\cQ}{\mathcal{Q}}
\newcommand{\cR}{\mathcal{R}}
\newcommand{\cN}{\mathcal{N}}

\newcommand{\cF}{\mathcal{F}}

\newcommand{\cT}{\mathcal{T}}
\newcommand{\cK}{\mathcal{K}}

\newcommand{\ccP}{\mathcal{P}}
\newcommand{\ccS}{\mathcal{S}}
\renewcommand{\dd}{\mathbf{d}}
\renewcommand{\(}{\left(}
\renewcommand{\)}{\right)}
\newcommand{\fS}{\ensuremath{\mathfrak{S}} }

\renewcommand{\l}{\lambda}


\title{The Poisson Tail Conjecture for Primes in Short Intervals}

\author[A.~Jha]{Abhishek Jha}

\address{Department of Mathematics, 1409 West Green Street,
 University of Illinois, Urbana-Champaign, Urbana, IL 61801, USA}

\email{jha33@illinois.edu}

\usepackage{microtype}
\begin{document}
\setcounter{tocdepth}{1}
\begin{abstract}
In 1976, Gallagher showed that, conditional on the Hardy--Littlewood conjectures, the number of primes below $x$ in a randomly chosen short interval of length $\lambda \log x$ asymptotically follows a Poisson distribution with mean $\lambda$. Correspondingly, the normalized gaps between consecutive primes follow an exponential distribution, provided that the scaling parameter $\lambda$ is fixed. We investigate the validity and limitations of the associated folklore Poisson Tail Conjecture as $\lambda$ is allowed to grow. For slowly growing $\lambda$, and conditional on a strong variant of the Hardy--Littlewood conjectures, we establish asymptotics showing that the local counting statistics agree with these predictions. Furthermore, we identify the scale of the phase transition and establish the breakdown of these distributions for larger $\lambda$. The proof relies on a novel combination of extremal interval sieve estimates and concentration inequalities from probability.
\end{abstract}
\maketitle

\section{Introduction}\label{sec: Introduction}

The prime number theorem implies that the average gap between two consecutive primes $p_n, p_{n+1}$ is of size $ \log p_n$. However, the sequence of prime gaps exhibits extreme deviations from this average. At the lower end of the spectrum, celebrated breakthroughs \cite{Zhang, maynard, polymath} established the existence of bounded gaps. In contrast, the true magnitude of unusually large prime gaps remains elusive. Ford, Green, Konyagin, Maynard, and Tao \cite{tao2} showed that \[p_{n+1}-p_n\gg \frac{(\log p_n)(\log\log p_n) (\log\log\log\log p_n)}{\log\log\log p_n}\] for infinitely many $n$\footnote{A very recent manuscript of OpenAI reports a further improvement to this lower bound; see \cite{OpenAILargeGaps}.}. These results concern the extreme behavior of prime gaps. Despite the progress, an asymptotic for the proportion of prime gaps that are a constant times the average gap has not been rigorously established.

\medskip
In a seminal work, Cramér \cite{cramer} proposed a probabilistic model for the distribution of primes, treating their occurrences as independent Bernoulli trials. Specifically, the primality of integers $n\ge 3$ was modeled via a random set of integers $\mathcal{C}=\{C_1,C_2,\ldots\}$ generated by including each $n$ independently with probability $1/\log n$. Cramér's heuristic is expected to govern the statistical behavior of the primes,  predicting that the distribution of normalized prime gaps asymptotically follows an exponential distribution. In particular, for any fixed real number $\l>0$, the random model almost surely satisfies \be\label{exponen}\ds{\lim_{x\rightarrow\infty}}\frac{\big|\{C_{n}\le x: C_{n+1}-C_n\ge \l\log x\}\big|}{x/\log x}=e^{-\l}.\ee 

\medskip
\edit{Gallagher \cite{gallagher} demonstrated that the analog of \eqref{exponen} for primes can be rigorously deduced from the Hardy--Littlewood conjectures \cite{HL}.}  These conjectures assert that the asymptotic relation \be\label{eq:HL-asym}
|\{n\le x:n+h\in\mathcal{P}\text{~for all~}h\in\cH\}|
=\(\fS(\cH)+o(1)\) \int_{2}^{x} \frac{\dd t}{\log^{|\cH|} t}\quad(x\rightarrow\infty)
\ee holds for any finite set $\cH \subset \mathbb{Z}$, where $\mathcal{P}$ is the set of primes and $\fS(\cH)$ is the singular series given by
\be\label{eq:singular-series}
\fS(\cH)\defeq \prod_p\bigg(1-\frac{|\cH\bmod p|}p\bigg)
\bigg(1-\frac1p\bigg)^{-|\cH|}.
\ee
The left side of \eqref{eq:HL-asym}
is bounded if $|\cH\bmod p|=p$ for
some prime $p$, since then for every integer $n$,
{one has} $p\,\mid\,n+h$ for some $h\in \cH$. 
{In this case, $\fS(\cH)=0$. In \cite[Section 1]{ksound}, Soundararajan discusses in great detail the intuition behind these conjectures and the emergence of the singular series from local arithmetic constraints. Gallagher showed that under these conjectures, the distribution of primes in intervals of logarithmic size is asymptotically Poissonian. For fixed $\l>0$, suppose that \eqref{eq:HL-asym} holds, for every fixed $r\ge 1$, uniformly over $|\cH|=r$ and $\cH\subseteq [0,\l \log x]$. Then for every fixed $k\ge 0$, \be\label{interval}\ds{\lim_{x\rightarrow\infty}}\,\frac1{x}\big|\{n\le x:\pi(n+\l \log x)-\pi(n)=k\}\big|=\frac{e^{-\l}\l^{k}}{k!}.\ee \edit{Gallagher's} analysis relies on the relation \be\label{singular}\sum_{\substack{\cH\subset [0,y]\\ |\cH|=k}}\fS(\cH)\sim \binom{y}{k}\quad(y\rightarrow\infty),\ee so that the singular series for sets of size $k$ have an average value of one.  It is not hard to see that this implies the analog of \eqref{exponen} for primes (see \cite[Theorem 2]{large} or \cite[Exercise 1.3]{ksound} for a proof of this equivalence). 

\medskip
Our discussion so far has focused exclusively on fixed $\l$. A fundamental question is how these local distributions behave when $\l$ grows with $x$. A well-documented flaw in Cramér's model is the failure of the analog of \eqref{eq:HL-asym} for all $\cH$. In particular, \edit{$\mathcal{C}$ almost surely satisfies an analog  of \eqref{eq:HL-asym} with $\fS(\cH)$ replaced by $1$; for example, $\mathcal{C}$  has $\sim x/\log^2 x$ pairs of consecutive integers below $x$.} The disparity arises from local arithmetic constraints: for any prime $p$, the set of primes completely avoids the residue class $0\bmod{p}$ (with the single exception of $p$ itself), whereas $\mathcal{C}$ is uniformly distributed modulo ${p}$. By incorporating these constraints, refined probabilistic models \cite{granville2,ford2025} recover the Hardy--Littlewood conjectures. Nonetheless, \edit{Cramér's heuristic predicts \eqref{interval} for a wide range of $\l$, because the averages of $\fS(\cH)$ and $1$ are asymptotically the same.} A standard application of the Borel--Cantelli lemma and Chebyshev's inequality shows that for any growing parameter $\l=o(\log x)$, with probability $1$,  we have \be \sum_{\substack{C_n\le x \\ C_{n+1}-C_n\ge \l \log x}}1\sim \(1-\frac1{\log x}\)^{\l \log x}\frac{x}{\log x}\sim e^{-\l}\frac{x}{\log x}\quad(x\rightarrow\infty).\ee This prediction naturally leads to the folklore Poisson Tail Conjecture \cite{large}.

\smallskip
\begin{conj}\label{taill}
    For any $\v>0$ and $0\le \l \le (\log x)^{1-\v}$, we have 
    \be\label{tail}\sum_{\substack{p_n\le x \\ p_{n+1}-p_n\ge \l \log x}}1\asymp e^{-\l}\frac{x}{\log x}.\ee
\end{conj} 

Our main goal in this paper is to determine how far these Poisson statistics persist as $\l$ grows with $x$. Conditional on a sufficiently uniform form of \eqref{eq:HL-asym}, we show that the Poisson and exponential distributions remain valid for a growing range of $\l$, while breaking down for larger values. In particular, our results locate the transition on the scale \[\exp\(\Theta\(\sqrt{\log_2 x\log_3 x}\)\).\]
Building upon Gallagher's work, Leung \cite{leung} proved that for any disjoint finite intervals $I_1,\ldots, I_r\subset (0,\infty)$, the number of primes in the intervals $n+I_1\log n,\ldots,n+I_r\log n$ is asymptotically jointly Poisson with parameters given by the lengths of $I_1,\ldots, I_r$.} More recently, Kravitz, Woo, and Xu \cite{Xu} established an averaged polynomial analog of Conjecture \ref{taill}. By averaging over families of polynomials $g \in \mathbb{Z}[x]$ of bounded degree and height, they proved that for almost all such polynomials, the number of prime values $g(n)$ for integers $n\in [x, x+L]$ asymptotically follows a Poisson distribution, confirming the polynomial analog of the Poisson Tail Conjecture in an averaged sense.

\medskip
A natural course to convert Conjecture \ref{taill} into a rigorous conditional statement is to extend Gallagher's approach employing the Hardy--Littlewood conjectures. \edit{However, a major challenge is that standard methods for estimating these sums of singular series are limited to ranges where $\l$ is at most $\log\log x$ (see \cite{sound}).} Recently, Kuperberg \cite{kuperberg} \edit{pushed} these singular series estimates to their limits, extracting new information regarding the tails of the distribution of primes. 
\edit{To overcome the technical barriers of the aforementioned approach, Banks, Ford, and Tao \cite{ford2025} proposed an alternative model that treats the primes as survivors of a random sieving process.} By probabilistically reinterpreting the summation in \eqref{singular}, their approach accommodates substantially larger values of $\l$, \edit{avoiding the need to analyze fluctuations in $\fS(\cH)$ directly.}

\medskip
To motivate their model, Banks, Ford, and Tao first reinterpret \eqref{eq:HL-asym} probabilistically. They relate the classic counting function to a random sieving process by treating the singular series as a product of local densities. For every prime $p$, let $\a_{p} \bmod{p}$ be a residue class chosen uniformly at random, with these choices being jointly independent. \edit{Define} the randomly sifted set $\ccS_{z}$ as \[\ccS_z\defeq \mathbb{Z}\setminus\bigcup_{p\le z}(\a_{p}+p\Z).\]Then, \edit{for admissible $\cH$,} \eqref{eq:HL-asym} takes the new form \be\label{probab}|\{n\le x:n+h\in\mathcal{P}\text{~for all~}h\in\cH\}|
\sim \int_{e^{2}}^{x} \P\(\cH\subset \ccS_{z(t)}\)\dd t  =\int_{e^2}^{x}\prod_{p\le z(t)}\(1-\frac{|\cH \bmod{p}|}{p}\)\dd t,\ee where $z(t)$ is the largest prime number such that $\prod_{p\le z(t)}(1-1/p)^{-1}\le \log t$ and $z(t)\sim t^{1/e^{\gamma}}$.  This formulation asserts that the probability of a random shift of $\cH$ lying in the primes is asymptotically equivalent to the probability of $\cH$ surviving this random sieve. This equivalence leads them to define their random set of integers $\cR$ as \[\cR\defeq \{n\ge e^2: n\in \ccS_{z(n)}\}.\] Crucially, the techniques developed in their analysis provide a probabilistic framework for addressing the Poisson Tail Conjecture (Conjecture~\ref{taill}) beyond the range accessible to direct estimates for sums of singular series. However, our arguments require a careful derivation of concentration estimates for the random sieve, together with suitable extremal interval sieve estimates. Unlike Cramér's model, the random set $\cR$ does satisfy the analog of \eqref{eq:HL-asym}, uniformly in a wide range of $\cH$ and with power-saving error terms. 

\smallskip
\begin{prop}[{\cite[Theorem 1.3]{ford2025}}]\label{eq:HL-R}
   Fix $c\in[1/2,1)$.
 Almost surely, we have
\[
|\{n\le x:n+h\in\cR\text{~for all~}h\in\cH\}|
=\fS(\cH)\int_{2}^{x} \frac{\dd t}{\log^{|\cH|} t}+
\O\big(x^{1-\frac{1-c}{8c^2-2c}+o(1)}\big)
\]
uniformly for all admissible tuples $\cH$ satisfying
$|\cH| \le \log^c x$ and in the range 
$\cH \subset \left[0,\exp\( \frac{\log^{1-c} x}{\log_2 x} \)\right]$ where the constant implied by the $\O$-symbol exists almost surely, though it is not uniformly bounded with respect to $x$. 
\end{prop}  For comparison, it has been conjectured that a much stronger version of \eqref{eq:HL-asym} holds (see, e.g., \cite{sound, kuperberg}), namely: 
\be|\{n\le x:n+h\in\mathcal{P}\text{~for all~}h\in\cH\}|
=\fS(\cH) \int_{2}^{x} \frac{\dd t}{\log^{|\cH|} t}+\O(x^{1/2+o(1)})\quad(x\rightarrow\infty).\ee This matches the error term in Proposition \ref{eq:HL-R} when $c=1/2$. 

\medskip
We are now in a position to state our main results. Let $\cA \subset \N$ be a set of positive integers. Put \be\label{r}M_{\cA}(x;y)\defeq \big|\{n\le  x:n\in \cA\text{ and }[n+1,n+y]\cap\cA=\varnothing\}\big|,\ee and \be\label{Gall} N_{\cA}(x;y,k)\defeq \big|\{n\le  x:|[n+1,n+y]\cap\cA|=k\}\big|.\ee  

Note the intentional asymmetry between the definitions of $M_{\cA}$ and $N_{\cA}$. \edit{For the quantity} $M_{\cA}(x;y)$, the index $n$ is restricted to $ \cA$ to count the gaps between consecutive elements of the set that exceed length $y$. On the other hand, $N_{\cA}(x;y,k)$ drops this restriction, measuring the frequency of short intervals $[n+1,n+y]$ containing exactly $k$ elements of $\cA$, sampled uniformly over all integers $n\le x$.
\smallskip
\begin{defn}
    Let $\phi$ and $\psi$ be positive real-valued functions. We say that $\cA\subset \N$ satisfies $\mathbf{HL}[\phi,\psi]$ if the relation
\be\label{eq:HLA}
|\{n\le x:n+h\in\cA\text{~for all~}h\in\cH\}|
=\fS(\cH) \int_{2}^{x} \frac{\dd t}{\log^{|\cH|} t}+
\edit{E(x,\cH)}
\ee
holds uniformly over all finite sets of integers $\cH \subset [0,\phi(x)]$
with $|\cH| \le \psi(x)$, \edit{and the error term satisfies $|E(x,\cH)|\le xe^{-\psi(x)\,\log_2 x}$ for $x\ge x_0$, where $x_0$ is some constant independent of $\cH$.}
\end{defn} Under the assumption that $\cA=\ccP$, this condition recovers the classical Hardy--Littlewood conjecture for primes. For convenience, we introduce some notation used throughout. We let $\log_k x$ denote the $k-$fold iterated logarithm truncated at zero; that is $\log_1 x\defeq \max\{0,\log x\}$ and $\log_k x\defeq \max\{0,\log(\log_{k-1} x)\}$ for $k\ge 2$. Let $\eta(x)>0$ tend to zero sufficiently slowly, and define
\be\label{eq:lambda_thresholds}
\l'=\lambda'(x) \defeq  \exp\left( \frac{1}{\sqrt{7}} \sqrt{\log_2 x \log_3 x} \right),\quad\textrm{and}\quad \l''=\lambda''(x) \defeq  \exp\left( \(\sqrt{2}+\eta(x)\) \sqrt{\log_2 x \log_3 x} \right).
\ee

Assuming that $\cA$ satisfies $\mathbf{HL}[\phi,\psi]$ for appropriate functions $\phi$ and $\psi$, our first result establishes a partial form of the asymptotic predicted in \eqref{tail}.

\smallskip
\begin{thm}[Uniform exponential distribution]\label{thm:HLgaps}
 Assume that
 $\cA \subset \N$ satisfies \[{\mathbf{HL}}\left[ \lambda'\log x, \exp(\sqrt{\log_2 x\log_3 x})\right].\] Then, for all sufficiently large $x$, we have 
\[
M_{\cA}(x;\lambda\log x)=\frac{xe^{-\lambda}}{\log x}\(1+\O\(\(\log_2 x\)^{-10}\)\),
\]
\edit{uniformly for $\,0<\l\le\l'$.} 
\end{thm} 

Before stating our next result, we recall a recent conjecture of Kuperberg \cite[Conjecture 1.7]{kuperberg} concerning a uniform variant of \eqref{interval}.

\begin{conj}\label{conj}
    Let $\l=(\log x)^{o(1)}$ and $k\ll\, (\log_2 x)^2$. Then, \be\label{interval2}\big|\{n\le x:\pi(n+\l \log x)-\pi(n)=k\}\big|\sim x\frac{e^{-\l}\l^{k}}{k!}\quad(x\rightarrow\infty).\ee
\end{conj}

In a similar spirit, our next result establishes a uniform variant of \eqref{interval} over a substantially wider range of $k$.

\smallskip
\begin{thm}[Uniform Poisson statistics]\label{thm:HLpoisson}
Assume that
 $\cA \subset \N$ satisfies \[{\mathbf{HL}}\left[ \lambda'\log x, \exp(\sqrt{\log_2 x\log_3 x})\right].\] Then, for all sufficiently large $x$, we have 
\[
N_{\cA}(x;\lambda\log x,k)=x\frac{e^{-\lambda}\l^{k}}{k!}\(1+\O\(\(\log_2 x\)^{-10}\)\),
\]
uniformly for $ 1/{\sqrt{\log x}}\le \lambda\le  \l'$ and  non-negative integers $k$ such that  \[\max\{k,k\log(k/\l)\}\le \exp\(\frac1{\sqrt{15}}\sqrt{\log_2 x \log_3 x}\).\]
\end{thm}

\begin{rem}
    We observe that Kuperberg restricted the formulation of the conjecture to the regime $k\ll (\log_2 x)^2$. Our theorem shows that, under the relevant assumptions, the predicted asymptotic behavior holds over a substantially wider range of $k$. 
\end{rem}

\begin{rem}
    We have not attempted to optimize the lower restriction $\l\ge 1/\sqrt{\log x}$, since our primary focus is on growing $\l$. Additionally, for small $\l$, the Brun--Titchmarsh theorem imposes a barrier on the parameter $k$, restricting it to the range $k\ll \frac{\l \log x}{\log (\l\log x)}$.
\end{rem}In contrast to our previous results, the following two results establish that the Poisson distribution breaks down as $\l$ grows. The proofs of these theorems rely on the combination of the oscillations of the sifting function with the probabilistic setup developed for the preceding results. To formulate them precisely, we require a definition related to the interval sieve. The basic interval sieve function is defined as \[S(x,y,z)=\#\{x<n\le x+y: n\text{ has no prime factor }\le z\}.\] Consider the minimum value \[S^{-}(y,z)=\displaystyle\min_{x}S(x,y,z).\] For our purposes, it is more convenient to utilize the following alternative formulation of $S^{-}(y,z)$. In particular, \[S^{-}(y,z)={\ds{\min_{(\a_p)}}\,|(0,y]\cap\ccS_{z}|}.\]  The special case where $z$ is a fixed power of $y$ is of basic interest. The sieve bounds of Jurkat and Richert \cite{sieve1} imply that for any fixed $v\ge 1$ we have \[(1+o(1))e^{-\gamma}F(v)z^v/\log z\ge S^{-}(z^v,z)\ge (1-o(1))e^{-\gamma}f(v)z^v/\log z\quad(z\rightarrow\infty),\] \smallskip where $\gamma$ is the Euler-Mascheroni constant and $f, F$ are the lower and upper bound linear sieve functions defined by the coupled differential-delay equations \alg{&f(v)=0,\quad F(v)=\frac{2e^{\gamma}}{v}\quad(0<v\le 2),\\&f(v)=\frac1{v}\int_{1}^{v-1}F(t)\,\dd t,\quad F(v)=\frac{2e^{\gamma}}{v}+\frac1{v}\int_1^{v-1}f(t)\,\dd t\quad (v\ge 2).}These functions are discussed at length in \cite{sieve1}. In particular, $f(v)<1<F(v)$ for all $v$, and $f(v)$ and $F(v)$ both tend to $1$ rapidly as $v\rightarrow\infty$. Following \cite{Hilde}, we define \[f^{+}(v)\defeq \limsup_{z\to\infty}\frac{S^{-}(z^v,z)}{e^{-\gamma}z^v/\log z}\quad\text{and}\quad f^{-}(v)\defeq \liminf_{z\to\infty}\frac{S^{-}(z^v,z)}{e^{-\gamma}z^v/\log z}\quad(v>1).\] We also denote $f^{+}(v+)\defeq\lim_{\v\to0}f^{+}(v+\v)$.

\medskip
 It is immediate from the definitions that $f^{+}(v)\ge f^{-}(v)\ge f(v)$. Interestingly, Granville \cite[Corollary 1]{granville} proved that the lower bound is sharp, showing $f^{-}(v)=f(v)$ conditional on the existence of an infinite sequence of Siegel zeros. In the other direction, considering random choices for the residue classes $\a_p\bmod p$ gives the upper bound $f^{+}(v)\le 1$. We record the following proposition regarding $f^{+}$, which shows that this inequality is strict. \edit{The proof is essentially due to Maier \cite{maier}; the details of which we defer to Section \ref{prelim}.}

 \smallskip
\begin{prop}\label{f+}
    The function $f^{+}$ is non-decreasing on $(1,\infty)$ and $f^{+}(v)<1$ for every $v>1$.
\end{prop}

 Before addressing the breakdown of the Poisson statistics, we recall some conjectures concerning the maximal gaps between primes. It was originally conjectured by Cram\'er that \[\ds{\limsup_{n\rightarrow{\infty}}}\,\frac{p_{n+1}-p_n}{\log^2 p_n}=1,\] \edit{the analog for his random set $\mathcal{C}$ holding almost surely.} However, Granville later argued this to be false by modifying Cramér's model and predicting that \[\ds{\limsup_{n\rightarrow{\infty}}}\,\frac{p_{n+1}-p_n}{\log^2 p_n}\ge 2e^{-\gamma}=1.12292\ldots.\] Most recently, Banks, Ford, and Tao \cite{ford2025}, as well as Granville and Lumley \cite{lumley}, refined Granville's conjecture. 

\smallskip
\begin{conj}\label{gapl}
    We have 
    \[\ds{\limsup_{n\rightarrow{\infty}}}\,\frac{p_{n+1}-p_n}{\log^2 p_n}=\frac1{f^{-}(2)}\ge 2e^{-\gamma},\]
    
    \smallskip 
    \noindent where the last inequality is true since $f^{-}(2)\le e^{\gamma}\,\omega(2)=e^{\gamma}/2$. It is a folklore conjecture that this inequality is, in fact, an equality. As stated before, under the assumption of Siegel zeros, one can show that $f^{-}(2)=0$, which would imply that the limit supremum diverges to infinity.
\end{conj} \edit{Assuming that this conjecture holds, for any constant $c\in (1
,2e^{-\gamma})$, there exists an infinite sequence of indices $(n_i)$ such that $p_{n_{i}+1}-p_{n_i}\ge c\log^2 p_{n_i}$. Choosing $x=p_{n_i}$ guarantees that for arbitrarily large $x$, there exists a prime gap strictly below $x$ of size at least $c\log^2 p_{n_i}\sim c\log^2 x$. For infinitely many $x$, this already contradicts the asymptotic relation in \eqref{tail} with $\l=c\log x$. Our goal here is to show that the breakdown occurs at much smaller values of $\l$.  We now state our result showing that \eqref{tail} and \eqref{interval2} fail to hold when $\l$ is a fixed small power of $\log x$.} 
\begin{thm}\label{UK4}
     Fix $c,c'\in (0,1)$ with $c<c'$. Assume that
 $\cA \subset \N$ satisfies \[{\mathbf{HL}}\left[ \log^{1+c'} x,\log^{c'} x\right].\] Set $\l\defeq \lambda(x)=\log^c x$. Then we have 
 \[
\frac{x}{\log x}\exp\Big(-\big(f(1+1/c)+o_{c,c'}(1)\big)\lambda\Big)\ge M_{\cA}(x;\l \log x)\ge\frac{x}{\log x}\exp\Big(-\big(f^{+}(1+1/c+)+o_{c,c'}(1)\big)\lambda\Big),\] as $x\rightarrow\infty$. Moreover, the same bounds hold for $N_{\cA}(x;\l \log x,0)$ as well.
           
\end{thm}

\begin{rem}
    The right-hand limit in the lower bound arises because the extremal sieve estimate in the proof is applied at $v>1+1/c$ and $v$ approaches $1+1/c$ from above. Since $f^{+}$ is non-decreasing, this limit exists and $f^{+}(1+1/c+)<1$ by Proposition \ref{f+}.
\end{rem}

\smallskip
Since $f(v)\le f^{+}(v)\le 1$ and $f(v)\rightarrow1$ as $v\rightarrow\infty$, we have $f^{+}(1+1/c)\rightarrow 1$ as $c\rightarrow0$. Hence, to detect deviations from the Poisson statistics when $\lambda$ grows slower than any fixed power of $\log x$, a different approach is required. The following theorem shows that the Poisson statistics already break down when $\l\ge \l''$.
 \begin{thm}\label{thm:HLasymp}
     Fix $c\in (0,1)$. Assume that
 $\cA \subset \N$ satisfies \[{\mathbf{HL}}\left[ \log^{1+c} x,\log^{c} x\right].\] Let $\l=\l(x)\ge \l''$ (as defined in \eqref{eq:lambda_thresholds}) be a parameter such that $\l=(\log x)^{o(1)}$ as $x\to \infty$. Then, for all sufficiently large $x$, we have  \[ M_{\cA}(x;\l \log x) \ge xe^{-\l}\exp\bigg[\l\exp\Big(-\big(1/\sqrt{2}+o(1)\big)\sqrt{\log_2 x\log_3 x}\Big)\bigg].\] 
 Moreover, the same bound holds for $N_{\cA}(x;\l\log x,0)$.
 \end{thm}
  \noindent Under the Hardy--Littlewood hypothesis for the primes, Theorem \ref{thm:HLasymp} shows that Conjecture \ref{conj} does not hold over its full conjectured range. 
\smallskip
 \begin{rem}\label{1.4}
     Heuristically, one expects the asymptotic of $M_{\cA}(x;\l \log x)$ to include an additional density factor of $1/\log x$ relative to  $N_{\cA}(x; \l \log x,0)$, reflecting the density of the initial element in $\cA$. However, in both Theorems \ref{UK4} and \ref{thm:HLasymp}, this logarithmic factor is absorbed by the large exponential error terms, making the bounds for $M_{\cA}(x;\l \log x)$ and $N_{\cA}(x;\l \log x,0)$ identical up to the stated error.

     Furthermore, Theorem \ref{thm:HLasymp} explicitly shows the breakdown of the Poisson statistics for $N_{\cA}$ and the corresponding exponential gap distribution for $M_{\cA}$. The Cram\'er--Gallagher heuristic predicts a density strictly proportional to $e^{-\l}$. In contrast, our lower bound has the additional multiplicative factor of the order \[\exp\bigg[\l\exp\Big(-\big(1/\sqrt{2}+o(1)\big)\sqrt{\log_2 x\log_3 x}\Big)\bigg],\] which diverges to infinity as $x\rightarrow\infty$. Thus, the frequency of empty intervals in both settings exceeds the heuristic prediction.
  \end{rem}

\subsection{Discussion of the Main Results}\label{disc}
Before proceeding to the proofs, we briefly remark on a few technical aspects and the underlying assumptions.

\medskip
\emph{Optimizing the Hardy--Littlewood hypothesis:} The assumptions underlying our results, specifically the conditions on the function $\psi$ in the hypothesis $\mathbf{HL}[ \phi, \psi]$, can be sharpened. The admissible range for $|\cH|$ can be made fully explicit in terms of $\l$ and weakened, as is implicitly shown within the proofs of these theorems. Similarly, the error terms in our asymptotic formulas could be optimized. However, we have chosen to prioritize the clarity of exposition and transparency of the theorem statements over extracting the strongest possible estimates.

\medskip
\emph{Uniformity in the Hardy--Littlewood conjectures:} As illustrated by Proposition \ref{eq:HL-R}, it is expected that a wide class of random integer sequences obeys Hardy--Littlewood type conjectures with a level of uniformity that far exceeds the assumptions required by our theorems. In particular, this strong degree of uniformity, capable of accommodating substantially larger tuples and shift ranges, is expected to hold for the sequence of prime numbers.

\medskip
\emph{Deriving the exponential distribution:} For a fixed parameter $\lambda$ as in \eqref{exponen}, the classical exponential gap distribution is known to follow directly from the Gallagher--type singular series averages in \eqref{singular}. However, in our regime where $\lambda$ is allowed to grow with $x$, transitioning from the Gallagher-type result to the gap distribution does not seem feasible. Specifically, as noted in Remark \ref{1.4}, evaluating $M_{\mathcal{A}}$ involves an additional global density factor of $1/\log x$. Consequently, establishing the asymptotics for the gap distribution $M_{\mathcal{A}}$ requires arguments slightly different from the techniques used to evaluate the interval counting function $N_{\mathcal{A}}$.

\medskip
\emph{Comparison with prior work:} Our approach draws on the methods of \cite[Section 8]{ford2025}, but there are several subtle differences. In \cite{ford2025}, the authors deal with the parameter $\l$ growing like $\log x$, which allows them to bound the probabilities of exceptional sets in the random sieving process. In our case, however, we deal with much smaller values of $\l$, forcing a more delicate treatment to bound the probabilities of these sets. Furthermore, our analysis of the gap distribution $M_{\cA}$ requires enforcing the condition $0\in \ccS_z$ throughout the sieving process. This breaks the pure martingale structure of the randomly sieved sets. To circumvent this obstruction, we develop a generalized variant of Azuma's inequality (Lemma \ref{lem:Azuma-ineq}) to handle these restricted sets. Most importantly, the focus in \cite[Section 8]{ford2025} is primarily on establishing the existence of large prime gaps. This task requires only a positive lower bound for the counting function $N_{\cA}(x;\l \log x,0)$. On the other hand, the proofs of Theorems \ref{thm:HLgaps} and \ref{thm:HLpoisson} require a more careful analysis to bound the error terms originating from the sieving process.

\subsection{Plan of the paper} The remainder of this paper is organized as follows. In Section \ref{prelim}, we collect several foundational results from sieve theory and probability, including an extremal interval sieve estimate and a generalized version of Azuma’s inequality. Section \ref{sieve} is devoted to the random sieving process; here, we establish the basic probabilistic and combinatorial estimates required for the various regimes of $\lambda$. In Section \ref{section1}, we use a Brun-type sieve to prove Theorems \ref{thm:HLgaps} and \ref{thm:HLpoisson}, which establish the validity of the exponential and Poisson distributions for small $\lambda$. In Section \ref{section2}, we provide the proofs of Theorems \ref{UK4} and \ref{thm:HLasymp}, detailing the breakdown of these distributions in the large $\lambda$ regimes. Finally, Section \ref{6} discusses the technical barriers that prevent further improvements to our bounds, along with concluding remarks regarding the methods.

\subsection{Acknowledgments} The author thanks his advisor, Kevin Ford, for suggesting this problem and for many helpful discussions. The author also thanks Jaya M. Manthripragada for helping him understand various probabilistic ideas and formulate the arguments in Section \ref{rapid}.
During the preparation of this work, the author was supported in part by the National Science Foundation under grant DMS-2301264.
\section{Preliminaries}\label{prelim}

\subsection{Notation} 
We largely retain the probabilistic setup of \cite{ford2025}. The indicator function of any set $\cT$ is denoted $\mathds{1}_{\cT}(n)$. We select residue classes $\mathbf{a}_p\bmod p$ uniformly and independently at random for each prime $p$, and then for any set of primes $\mathcal{Q}$ we denote by $\cF_{\mathcal{Q}}$ the ordered tuple $(\mathbf{a}_p: p\in \mathcal{Q})$; often we condition our probabilities on $\cF_{\mathcal{Q}}$ for a fixed choice of $\mathcal{Q}$. In a similar vein, we define $\cF^{0}_{\mathcal{Q}}$ by assuming that $\mathbf{a}_p \neq 0$ for all $p\in \mathcal{Q}$; that is, selecting $\a_{p}$ uniformly from $\{1,\ldots,p-1\}$.

Probability and expectation are denoted by $\P$ and $\E$ respectively. We use $(\P_{\mathcal{Q}},\E_{\cQ})$ to denote the probability and expectation with respect to random $\cF_{\cQ}$.  When $\cQ$ is the
set of primes in $(c,d]$, we write 
$\cF_{c,d}$, $\P_{c,d}$ and $\E_{c,d}$; if $\cQ$
is the set of primes $\le c$, we write $\cF_c$, $\P_c$ and $\E_c$.
In particular, $\P_{c,d}$ refers to the probability
over random $\cF_{c,d}$, often with conditioning on $\cF_{c}$. Analogously, we write $(\P^0_{\mathcal{Q}},\E^0_{\cQ})$ to denote the probability and expectation with respect to random $\cF^0_{\cQ}$.

The symbol $\gamma$ is reserved for the Euler-Mascheroni constant. Implied constants in the standard asymptotic notations $\mathcal{O}$, $\ll$, $\gg$, and $\asymp$ are absolute unless otherwise specified. The notation $o(1)$ is used to indicate a function that tends to zero as $x \to \infty$; in expressions like $1-o(1)$, the $o(1)$ is assumed to be positive, and $F \sim G$ means $F = (1+o(1))G$.

\subsection{Results from Sieve and Probability Theory}

We collect here some standard results from sieve theory and probability that are used in the rest of the paper.
We begin by stating the fundamental lemma of the combinatorial sieve (see\cite[Theorem 6.12]{opera}), followed by the required lower and upper bound sieve estimates. Given a finite set of integers $\mathcal{A}$ and a finite set of primes $\mathcal{P}$, we define
$$
S(\mathcal{A}, \mathcal{P}) \defeq  \#\{a \in \mathcal{A} : (a, \mathcal{P}) = 1\},
$$
where the notation $(a, \mathcal{P}) = 1$ is shorthand for $(a,\prod_{p\in \ccP}p)=1$, denoting that $a$ has no prime factors from $\mathcal{P}$. Furthermore, we use the notation $d \mid \mathcal{P}$ to denote $d \mid \prod_{p \in \mathcal{P}} p$.

\smallskip
\begin{lem}[Fundamental lemma of sieve theory]\label{lmsieve}
Assume that for a finite set of integers $\mathcal{A}$ and a finite set of primes $\mathcal{P}$, there exist a nonnegative multiplicative function $g(d)$, a parameter $y$, and positive constants $\alpha$ and $C>1$ such that:

 $\bullet\,\,$ for any $d \mid \mathcal{P}$, \[|\{\mathfrak{a} \in \mathcal{A} : d \mid \mathfrak{a}\}| = \frac{g(d)}{d} \cdot {y} + r_d,\]

$\quad\quad$ and \[g(p) < p \quad \text{for all } p \in \mathcal{P},\]

 $\bullet\,\,$ for any real $v>w\ge 2$, \[\prod_{w < p \leq v, p \in \mathcal{P}} \left(1 - \frac{g(p)}{p}\right)^{-1} \le C \left(\frac{\log v}{\log w}\right)^\alpha .\]

In particular, if we take $\ccP$ to be the set of primes $p\le z$ and let $D\ge z\ge 2$, then uniformly for $\mathcal{A}, y,$ and $u=(\log D)/\log z\ge 4\alpha+2$, we have
$$
S(\mathcal{A}, \mathcal{P}) = y\prod_{p \in \mathcal{P}} \left( 1 - {g(p)}/{p} \right) \Big(1 + \theta C^{\eta}\exp\big[-u\log u +u\log\log (uC) +\O(u) \big]\Big) + \O\Big( \sum_{d \mid \mathcal{P}, d \leq D} |r_d| \Big),
$$
where $|\theta|\le 1$, while $\eta$ and implied constants depend only on $\alpha$.
\end{lem}

\smallskip
\begin{lem}[Lower bound sieve,{\cite[Theorem 5]{sieve1}}]\label{lem:sieve-lower}
    Assume that for a finite set of integers $\mathcal{A}$ and a finite set of primes $\mathcal{P}$, there exists a parameter $y$ such that, for every $d\mid \ccP$,
 \[\big|\#\{\mathfrak{a} \in \mathcal{A} : d \mid \mathfrak{a}\} -{y}/{d}\big| \le 1.\]
Suppose further that $P$ consists of all primes $p\le z$. Then, for every fixed $\alpha>0$, uniformly for $\mathcal{A}\,\textrm{and}\, y$ in the range $u=(\log y)/\log z\ge 2+\alpha$, we have \[S(\cA,\ccP)\ge (1-o_{\alpha}(1))e^{-\gamma}f(u)\frac{y}{\log z}\quad(z\rightarrow\infty).\] In particular, we have $S(\cA,\ccP)\gg_{\alpha}y/\log z$ uniformly for $u\ge 2+\alpha$.
\end{lem}

\smallskip
\begin{lem}[Upper bound sieve,{\cite[Theorem 3.8]{mont}}]\label{lem:sieve-upper}
Suppose that $(a,q)=1$, $(P,q)=1$, and $x$ and $y$ are real numbers with $y\ge 2q$. We have \[\big|\{x<n\le x+y:\,n\equiv a\bmod{q},\, (n,P)=1\}\big|\le \frac{e^{\gamma}y}{q}\(\prod_{{p\mid P,\; p\le \sqrt{y/q}}}\(1-\frac1{p}\)\)\(1+\O\pfrac1{\log (y/q)}\).\]
\end{lem}

\begin{lem}[Brun's sieve]\label{lem:UK}
Suppose that $y \geq 1$. Let $\cN, \cA$ be sets of positive integers. Let $\Omega \subset[0,y]$ be a finite set of integers.  
For each $n\in \cN$, define the counting function \[v(n)\defeq \big|\{h\in \Omega:n+h\in\cA\}\big|.\]
\noindent \emph{(i) General Case (Generalization of \cite[Lemma 8.1]{ford2025}):}
 For an integer $k\ge 0$, define
\[
T\defeq \big|\{n\in \cN\,:\,v(n)=k\}\big|
\]
and
\[
U_K\defeq  \sum_{\ell=k}^K(-1)^{\ell-k}\binom{\ell}{k}\sum_{\substack{\cH\subset\,\Omega\\|\cH|=\ell}} \;
\sum_{n\in\cN} \;\prod_{h\in\cH} \mathds{1}_{\cA}(n+h)\qquad
(K\ge k).
\]
Then, for any $K$ with $K-k$ even we have $T\le U_K$,
and for $K-k$ odd we have $T\ge U_K$.

\vspace{0.5cm}
\noindent \emph{(ii) Fixed Point Case ($h=0$ is fixed):} Put
\[
T\defeq \sum_{n\in\cN\cap \cA}\prod_{h\in [1,y]}\(1-\mathds{1}_{\cA}(n+h)\)
\]
and
\[
U'_K\defeq  \sum_{\ell=0}^K(-1)^\ell\sum_{\substack{\cH\subset[1,y]\\|\cH|=\ell}} \;
\sum_{n\in\cN} \;\prod_{h\in\cH\cup \{0\}} \mathds{1}_{\cA}(n+h)\qquad
(K\ge 0).
\]
Then, for any even $K$ we have $T\le U'_K$,
and for any odd $K$ we have $T\ge U'_K$.
\end{lem}
\begin{proof}
It suffices to prove the first part, as the second part follows immediately as a special case by setting $k=0$, $\Omega=[1,y]\cap\ \Z$ and replacing $\cN$ with $\cN\cap\cA$. 
For any integers $K,k,u\ge 0$ let
\be\label{d2}
\delta_K(u,k)\defeq \sum_{\ell=k}^K(-1)^{\ell-k}\binom{\ell}{k}\binom{u}{\ell}\quad \text{and}\quad
\delta(u,k)\defeq \begin{cases}
1&\quad\hbox{if $u=k$},\\
0&\quad\hbox{if $u\neq k$}.
\end{cases}
\ee
Observe that
\be\label{d1}
\delta(u,k)\le\delta_K(u,k)\quad{\(K\equiv k\bmod{2}\)}\quad
\delta(u,k)\ge\delta_K(u,k)\quad{\(K\equiv k+1\bmod{2}\)};
\ee
since:
\alg{\delta(u,k)=\sum_{\ell=k}^{\infty}(-1)^{\ell-k}\binom{\ell}{k}\binom{u}{\ell}&=\delta_K(u,k)+\binom{u}{k}\sum_{\ell=K+1}^{\infty}(-1)^{\ell-k}\binom{u-k}{\ell-k}\\&=\delta_K(u,k)+(-1)^{K-k+1}\binom{u}{k}\,\binom{u-k-1}{K-k}.}We have
\[
T=\sum_{n\in\cN}\delta(v(n),k)=\sum_{n\in\cN} \delta_K(v(n),k)+\theta,
\]
where $\theta\le 0$ if $K\equiv k\bmod{2}$ and $\theta\ge 0$ otherwise. Also,
\[
\sum_{n\in\cN} \delta_K(v(n),k)
=\sum_{\ell=k}^K(-1)^{\ell-k}\binom{\ell}{k}\sum_{n\in\cN}\binom{v(n)}{\ell}=U_K
\]
since
\[
\binom{v(n)}{\ell}=\sum_{\substack{\cH\subset\,\Omega \\|\cH|=\ell}}
\;\prod_{h\in\cH}\mathds{1}_{\cA}(n+h)\qquad(n\in\cN),
\]
and the lemma is proved. 
\end{proof}

It will be convenient to express the combinatorial bounds of the previous lemma in probabilistic language for our arguments in later sections. 

\begin{lem}[Probabilistic Brun's sieve]\label{probru}
    Suppose $\cA\subset \N$ satisfies $\mathbf{HL}[\phi,\psi]$. Let $x$ be sufficiently large such that $\phi(x)\le \log^2 x$ and $\psi(x)\le (\log x)/10 \log_2 x$. Furthermore, let $6\le y\le \phi(x)$ and $y^{100}\le x'\le x$. Let $K$ be a positive integer such that \[4\le K\le \min\{\psi(x)/6,y/3\}.\] For $t\ge e^2$, put $S_{z(t)}\defeq |\ccS_{z(t)}\cap [1,y]|$, and for $n\in [x',x]$, define the counting function \[v(n)\defeq |\{h\in [1,y]:n+h\in \cA\}|.\] 

\smallskip\noindent \emph{(i)} For any non-negative integer $k\le K$, we have \alg{\sum_{\substack{x'< n\le x \\ v(n)=k}}1&=\int_{x'}^{x}\P_{z(t)}\(S_{z(t)}=k\)\,\dd t+ \O\(\int_{x'}^{x}(K+1)^{k}\,\E_{z(t)}\binom{S_{z(t)}}{K+1}\,\dd t\)\,+\O\(\frac{x}{\log^{K} x}+x'y^{2K+2}\).}

\noindent \emph{(ii)} We have 
\[\sum_{\substack{x'< n\le x \\n\in \cA,\, v(n)=0}}1=\int_{x'}^{x}\P\(0\in \ccS_{z(t)}\)\( \P_{z(t)}^{\,0}(S_{z(t)}=0) +\O\(
\E_{z(t)}^{\,0}\binom{S_{z(t)}}{K+1}\)\)\, \dd t+\O\(\frac{x}{\log^{K} x}+x'y^{2K+2}\).\]
 \end{lem}

 \begin{proof}
    The proofs for both parts are similar, but we include them both for completeness. We retain the setup of the previous lemma. Let \[T_k\defeq \sum_{\substack{x'< n\le x \\ v(n)=k}}1.\]
    
    For (i), it follows from Lemma \ref{lem:UK} (i) (with $\cN=[x',x]\cap \Z$ and $\Omega=[1,y]\cap \Z$) that $T_k$ is bounded between $U_K$ and $U_{K+1}$. Let $M\in \{K,K+1\}$. By the definition of $U_M$ in Lemma \ref{lem:UK} (i), we apply \eqref{eq:HLA} at $x$. The contribution of the initial interval $n\le x'$, along with the corresponding main-term integral over $[2,x']$, is
    \[\ll x'\sum_{\ell=k}^{M}\binom{\ell}{k}\binom{y}{\ell}(1+(C\log y)^{\ell})\ll x'y^{2K+2},\] where we used the bound $\fS(\cH)\ll (C \log y)^{|\cH|}$. We have 
    \be\label{11}
    \begin{aligned}U_M=\int_{x'}^{x}&\,\sum_{\ell=k}^{M}(-1)^{\ell-k}\binom{\ell}{k}\frac{1}{\log^{\ell} t}\sum_{\substack{\cH\subset [1,y]\\|\cH|=\ell}}\fS(\cH)\,\dd t \\&+\O\(\sum_{\ell=k}^{M}\binom{\ell}{k}\binom{y}{\ell}xe^{-\psi(x)\,\log_2 x}\)+\O\(x'y^{2K+2}\).
    \end{aligned}
    \ee  
   For the sum in the error term, the terms in the sum are maximized at $\ell=M$ since $M\le K+1\le y/2$. Bounding the sum by its largest term and using the assumption $\psi(x)\ge 6K$, the error is bounded by 
    \alg{&\ll M\binom{M}{k}\binom{y}{M}\frac{x}{(\log x)^{6K}}\le (K+1)\binom{K+1}{k}\binom{y}{K+1}\frac{x}{(\log x)^{6K}}\\[2ex]&< x(K+1)^{k+1}y^{K+1}\log^{-6K}x\le xy^{2K+2}\log^{-6K}x\\[2ex]&\le x\log^{-K}x,} where the last estimate follows from $y\le \log^2 x$ and $K\ge 4$.
    By the assumption $x'\ge y^{100}\ge (K+1)^{100}$ and \cite[Lemma 3.5]{ford2025}, replacing $\fS(\cH)/\log^{\ell} t$ with $V_{\cH}(z(t))$ induces a relative error of size $\O(t^{-0.54})$ for $t\in [x',x]$. Summing over all subsets $\cH$ and integrating over $t\in [x',x]$ in \eqref{11}, we bound the absolute error by \[\ll (K+1)^{k+1}\binom{y}{K+1}\int_{x'}^{x}\frac{\dd t}{t^{0.54}}\ll (K+1)^{k+1}y^{K+1}x^{0.46}\ll x\log^{-K}x.\] 
    By the linearity of expectation and the observation that $V_{\cH}(z(t))=\P\(\cH\subset \ccS_{z(t)}\)$, we infer 
    \begin{align} U_M&=\int_{x'}^{x}\E_{z(t)}\left(\sum_{\ell=k}^{M}(-1)^{\ell-k}\binom{\ell}{k}\binom{S_{z(t)}}{\ell}\right)\dd t\,+\O\(x\log^{-K}x+x'y^{2K+2}\)\\&\label{14}=\int_{x'}^{x}\E_{z(t)}\(\delta_M(S_{z(t)},k)\)\dd t\,+\O(x\log^{-K}x+x'y^{2K+2}).\end{align} 
    We now focus on the inner expectation. Substituting $u=S_{z(t)}$ in \eqref{d2} and \eqref{d1}, it is immediate that the indicator function $\delta(S_{z(t)},k)$ is bounded between $\delta_K(S_{z(t)},k)$ and $\delta_{K+1}(S_{z(t)},k)$. 
    Thus, \be\label{d4}|\delta_M(S_{z(t)},k)-\delta(S_{z(t)},k)|\le \binom{K+1}{k}\binom{S_{z(t)}}{K+1}.\ee Taking expectations, we deduce \be\label{13}\E_{z(t)}\(\delta_M(S_{z(t)},k)\)=\P_{z(t)}\(S_{z(t)}=k\)+\O\((K+1)^{k}\,\E_{z(t)}\binom{S_{z(t)}}{K+1}\).\ee Combining \eqref{14} and \eqref{13}, we conclude that \[U_M=\int_{x'}^{x}\P_{z(t)}\(S_{z(t)}=k\)\,\dd t+ \O\(\int_{x'}^{x}(K+1)^{k}\,\E_{z(t)}\binom{S_{z(t)}}{K+1}\,\dd t\)\,+\O\(\frac{x}{\log^K x}+x'y^{2K+2}\).\] Since $T_k$ is bounded between $U_K$ and $U_{K+1}$, we are done.

\medskip
    We now turn our attention to (ii). As before, let \[T\defeq \sum_{\substack{x'< n\le x \\ n\in \cA,\,v(n)=0}}1.\] Applying Lemma \ref{lem:UK} (ii) (with $\cN=[x',x]\cap \Z$), we get that $T$ is bounded between $U'_K$ and $U'_{K+1}$. Let $M\in \{K,K+1\}$. Again, using the definition of $U'_M$ and applying $\eqref{eq:HLA}$, we obtain \be
    \begin{aligned}
\label{15}
U'_M=\int_{x'}^{x}&\,\sum_{\ell=0}^{M}(-1)^{\ell}\frac{1}{\log^{\ell+1} t}\sum_{\substack{\cH\subset [1,y]\\|\cH|=\ell}}\fS(\cH\cup\{0\})\,\dd t \\&+\O\(\sum_{\ell=0}^{M}\binom{y}{\ell}xe^{-\psi(x)\,\log_2 x}\)+\O\(x'y^{2K+2}\).
    \end{aligned}
    \ee Arguing as before, the first error term is bounded above by $x\log^{-K}x$. Replacing $\fS(\cH\cup\{0\})/\log^{\ell+1} t$ with $V_{\cH\cup\{0\}}(z(t))$ induces an overall error of same size. Therefore, \eqref{15} simplifies to \be \label{16}U'_M=\int_{x'}^{x}\sum_{\ell=0}^{M}(-1)^{\ell}\sum_{\substack{\cH\subset [1,y]\\|\cH|=\ell}}V_{\cH\cup\{0\}}(z(t))\,\dd t+\O\(x\log^{-K}x+x'y^{2K+2}\).\ee We need to express the inner summation in terms of expectations. Consequently, we expand the inner summation as follows: \alg{
\sum_{\substack{\cH\subset[1,y]\\|\cH|=\ell}}V_{\cH\cup\{0\}}(z(t))&=\sum_{\substack{\cH\subset[1,y]\\|\cH|=\ell}}\P(\cH\cup\{0\}\subset \ccS_{z(t)})\\&=\sum_{\substack{\cH\subset[1,y]\\|\cH|=\ell}}\P\(\cH\subset\ccS_{z(t)}\,|\, 0\in \ccS_{z(t)}\)\cdot\P\(0\in \ccS_{z(t)}\)
\\&=\P\(0\in \ccS_{z(t)}\)\cdot \E_{z(t)}^{\,0}\binom{S_{z(t)}}{\ell}.
} This implies that \begin{align} U'_M&=\int_{x'}^{x}\P\(0\in \ccS_{z(t)}\)\sum_{\ell=0}^{M}(-1)^{\ell}\,\E_{z(t)}^{\,0}\binom{S_{z(t)}}{\ell}\,\dd t+\O\(x\log^{-K}x+x'y^{2K+2}\)\\&\label{18}=\int_{x'}^{x}\P\(0\in \ccS_{z(t)}\)\,\E_{z(t)}^{\,0}\(\delta_M(S_{z(t)},0)\)\,\dd t+\O(x\log^{-K}x+x'y^{2K+2}).\end{align} Using the bound in \eqref{d4} and taking the conditional expectation, we arrive at the following estimate. \be\label{17}\E_{z(t)}^{\,0}\(\delta_M(S_{z(t)},0)\)=\P_{z(t)}^{\,0}\(S_{z(t)}=0\)+\O\(\E_{z(t)}^{\,0}\binom{S_{z(t)}}{K+1}\).\ee Along with \eqref{18}, this leads to the asymptotic  \[U'_M=\int_{x'}^{x}\P\(0\in \ccS_{z(t)}\)\( \P_{z(t)}^{\,0}(S_{z(t)}=0) +\O\(
\E_{z(t)}^{\,0}\binom{S_{z(t)}}{K+1}\)\)\, \dd t+\O\(\frac{x}{\log^K x}+x'y^{2K+2}\).\] This completes the proof since $T$ is bounded between $U'_K$ and $U'_{K+1}$. 
\end{proof}

\begin{lem}[Extremal interval sieve,{\cite[Proposition 3]{maier}}]\label{maierlem}
    For sufficiently large $z$, let $y=y(z)$ satisfy \[z\le  y\le \exp(\sqrt{z}),\quad u\defeq\frac{\log y}{\log z}\rightarrow\infty.\]  We have   
    \[S^{-}(y,z)\le y\,\prod_{p\le z}(1-1/p)\Big(1-\exp\big[{-(1+o(1))u\log u}\big]\Big).\]
\end{lem}

Before proceeding, we supply the proof of the upper bound on the sifting function $f^{+}$ stated in the introduction. Recall Proposition \ref{f+}, which asserts that for all $v>1$, we have $f^{+}(v)<1$. 

\begin{proof}[Proof of Proposition \ref{f+}]
     First, we show that $f^{+}$ is non-decreasing. This is evident since if $A<B$, then any interval of length $z^{B}$ is the disjoint union of $\flr{z^{B-A}}$ intervals of length $z^{A}$ and one remaining interval of length at most $z^A$. By standard sieve estimates, we know that \[S(0,z^v,z)\sim \omega(v)\frac{z^v}{\log z},\] where $\omega$ is Buchstab's function, defined by $\omega(v)=1/v$ for $1\le v\le 2$, and $(v\,\omega(v))'=\omega(v-1)$ for all $v\ge 2$. By the monotonicity of $f^{+}$, we have $f^{+}(v)\le e^{\gamma}\ds{\min_{u\ge v}\,\omega(u)}.$ By \cite[Lemma 4]{maier2}, it is known that $\omega(v)-e^{-\gamma}$ has a sign change on every unit interval. Thus, we obtain $f^{+}(v)<1$.
\end{proof}
\begin{lem}\label{lem:tower}
Let $X$ be a finite-valued random variable and let $E$ be an event with $\P(E)>0$. For $0<a<1$, 
\[
\E\left[a^{X}\,|\,{E}\right]\ge a^{\E[X\,|\,{E}]}.
\]
\end{lem}
\begin{proof}
     The result follows from Jensen's inequality, since $x\to a^x$ is convex for $0<a<1$.
\end{proof}

\begin{lem}[Azuma's inequality]\label{lem:Azuma}
Suppose that $X_0,\ldots,X_n$ is a submartingale sequence taking values in a finite set of real numbers with $|X_{i+1}-X_i|\le c_i$
for each $0\le i\le n-1$.  Then for all $\epsilon>0$,
$$
\mathbb{P}(X_n - X_0 \le -\epsilon) \le \exp\left(\frac{-\epsilon^2}{2 \left(c_0^{2}+\cdots+c_{n-1}^{2}\right)}\right).
$$
And symmetrically (when the sequence is a supermartingale):
$$
\mathbb{P}(X_n - X_0 \ge \epsilon) \le \exp\left(\frac{-\epsilon^2}{2 \left(c_0^{2}+\cdots+c_{n-1}^{2}\right)}\right).
$$
If the sequence is a martingale, using both inequalities above and applying the union bound allows one to obtain a two-sided bound:
$$
\mathbb{P}(|X_n - X_0| \ge \epsilon) \le 2\exp\left(\frac{-\epsilon^2}{2 \left(c_0^{2}+\cdots+c_{n-1}^{2}\right)}\right).
$$ 
\end{lem}

\begin{lem}[Generalized Azuma's inequality]\label{lem:Azuma-ineq}
Suppose that $X_0,\ldots,X_n$ is a sequence of real-valued random variables taking only finitely many values and adapted to a filtration $(\mathcal{F}_i)$ such that \[X_i\le \E(X_{i+1}|\cF_i)\le X_i+d_i\] with $|X_{i+1}-X_i|\le c_i$ and $d_i\ge 0$ for each $0\le i\le n-1$. Then,
\[
\P\big( |X_n-X_0| \ge t \big) \le 2\exp\left\{-\frac{t^{2}}{8\left(c_0^{2}+\cdots+c_{n-1}^{2}\right)}\right\}\qquad(t> 2(d_0+\cdots +d_{n-1})),
\] 
\end{lem}

\begin{proof}
 For a fixed $\alpha> 0$, we consider the convex function $f(x) = e^{\alpha x}$.
For any $|x| \le c$, $f(x)$ is below the line segment from $(-c, f(-c))$ to $(c, f(c))$. In
other words, we have \[e^{\alpha x}\le \frac1{2c}(e^{\alpha c}-e^{-\alpha c})x +\frac1{2}(e^{\alpha c}+e^{-\alpha c}).\]

Thus, we get 
\begin{align*}\E\(e^{\alpha(X_{i+1}-X_{i})}\mid \cF_{i}\)&\le \E\left(\frac1{2c_{i}}(e^{\alpha c_i}-e^{-\alpha c_i})(X_{i+1}-X_{i}) +\frac1{2}(e^{\alpha c_i}+e^{-\alpha c_i})\mid \cF_{i}\right)\\&\le \frac{d_{i}}{2c_i}(e^{\alpha c_i}-e^{-\alpha c_i})+\frac1{2}(e^{\alpha c_i}+e^{-\alpha c_i})\\& =\frac{d_i}{c_i}\sinh(\alpha c_i)+\cosh(\alpha c_i).\end{align*} We know that $ \sinh(x) \le x\cosh(x)$ and $ \cosh(x)\le e^{x^{2}/2}$ for $x\ge 0$. Therefore, \[\E\(e^{\alpha(X_{i+1}-X_{i})}\mid \cF_{i}\) \le d_i \alpha \cosh(\alpha c_i)+\cosh(\alpha c_i)\le e^{d_i\alpha+ \alpha^2c_i^2/2}.\] Consider the moment generating function $\E\(e^{\alpha (X_{n}-X_0)}\)$. We will split this up as 
\alg{\E\(e^{\alpha (X_{n}-X_0)}\)&=\E\(e^{\alpha(X_{n}-X_{n-1})}\cdot e^{\alpha (X_{n-1}-X_0)}\)\\&=\E\(\E\(e^{\alpha (X_{n}-X_{n-1})}\mid \cF_{n-1}\)e^{\alpha (X_{n-1}-X_0)}\).}
This implies that 
\[\E\(e^{\alpha (X_{n}-X_0)}\)\le e^{d_{n-1}\alpha+\alpha^{2} c_{n-1}^{2}/2}\, \E\(\e^{\alpha (X_{n-1}-X_0)}\).\] Proceeding inductively, \[\E\(e^{\alpha (X_n-X_0)}\)\le e^{{  \sum_{i=0}^{n-1} (\alpha^{2}c_i^{2}}/{2}+d_i\alpha)}.\] By Markov's inequality on the moment generating function, we have:
\begin{align*}
    \P(X_n-X_0\ge t)&= \P\big(e^{\alpha(X_n-X_0)}\ge e^{t\alpha}\big)\\&\le e^{-\alpha t}\E\( e^{\alpha(X_n-X_0)}\)\\&\le e^{-\alpha t}e^{{  \sum_{i=0}^{n-1} (\alpha^{2}c_i^{2}}/{2}+d_i\alpha)}\\&=\exp\left\{{-\frac{\left(t-\sum_{i=0}^{n-1}d_i\right)^{2}}{2\sum_{i=0}^{n-1}c_i^{2}}}\right\}\\&\le \exp\left\{{-\frac{t^{2}}{8\left(\sum_{i=0}^{n-1}c_i^{2}\right)}}\right\},
\end{align*}
where we choose $\alpha=(t-\sum_{i=0}^{n-1}d_i)/{(c_0^2+\cdots+c_{n-1}^2)}$ to minimize the probability.

\smallskip
On the other hand, we know that the sequence $(X_{i})$ forms a submartingale. Applying Azuma's inequality for submartingales (Lemma \ref{lem:Azuma}), we get that \[\P\left(X_n-X_0\le{ -t}\right)\le \exp\left\{{-\frac{t^{2}}{2\left(c_0^{2}+\cdots+c_{n-1}^{2}\right)}}\right\}\qquad(t> 0).\] Combining both tail bounds, we obtain our result.
\end{proof}

\begin{rem}
    The terms $d_i$ account for the small positive drift that arises later in the sieving process when the pure martingale structure is lost, an obstruction we discussed earlier in Section \ref{disc}.
\end{rem}

\section{Random Sieving}\label{sieve}

Throughout the sequel, we employ the notation  $$\Theta_{z}\defeq \prod_{p\le z}\(1-\frac{1}{p}\)\qquad\text{  and   }\qquad\Theta_{x,y}\defeq \prod_{x< p \le y}\(1-\frac1{p}\)=\frac{\Theta_y}{\Theta_x},$$ along with \[\hat\Theta_{x,y}\defeq \prod_{x< p\le y}\left(1-\frac1{p-1}\right)=\frac{\hat\Theta_y}{\hat\Theta_x},\quad\text{where }\quad \hat\Theta_{z}\defeq \prod_{2<p\le z}\(1-\frac1{p-1}\).\] 

\vspace{0.2cm}
\noindent Throughout this section, we assume $x$ is sufficiently large and define \[y\defeq \lambda\log x\quad\( 1/\sqrt{\log x}\le \l\le \log x\).\] Furthermore, we assume that $k$ is a non-negative integer satisfying $\max\{k,k\log (k/\l)\}\le  \l'$.
\noindent  We define
\[
\ccS_w(y) \defeq  [1,y]\cap\ccS_w
\]
and when the value of $y$ is clear from context we put
\[S_w\defeq |\ccS_w(y)|.\] We denote \be\label{lk}\lambda_k\defeq \max\{\lambda,k, k\log (k/\l), (\log_2 x)^{200}\}.\ee We also define the $k=0$ analog as \be\label{l0}\lambda_0\defeq \max\{\lambda, (\log_2 x)^{200}\}.\ee  Our estimates depend crucially on the growth rate of $\l$ relative to $x$. To facilitate the exposition, we classify the various regimes of $\l$ used throughout the paper as follows.
\bigskip
\[
\begin{array}{|c|c|c|c|}
\hline

\textbf{Regime} & \textbf{Range of } {\lambda} & \textbf{Key Application} \\
\hline
\text{Arbitrary} & \lambda \ge 1/\sqrt{\log x} & \S \ref{rapid3} \\
\hline
\text{Slow Growth} & 1/\sqrt{\log x} \le \lambda \le \l' & \text{Lem. \ref{f1}--\ref{f4}, Cor. \ref{cor:w1w5}, \S \ref{section1}}  \\
\hline
\text{Rapid Growth} & \l=\log^c x \quad(0<c<1) & \text{Lem. \ref{f5}, \S \ref{rapid1} }  \\
\hline
\end{array}
\]
\smallskip
\subsection{Sieving for Small Primes (Using Fundamental Lemma)}\label{lsmall}

\begin{lem}\label{f1}
    Let $\l,k$ satisfy 
    \[\max\{k,k\log (k/\l)\}\le\exp\(\frac1{\sqrt{15}}\sqrt{\log_2 x\log_3 x}\),\quad \textrm{and}\quad\frac1{\sqrt{\log x}}\le \l\le \l'.\] We have\begin{equation}\label{31}
        S_{\l_k^{25/8}}=y\,\Theta_{\l_k^{25/8}}\left(1 + \mathcal{O}\left( 1/{\lambda_k^{21/20}}\right)\right). 
    \end{equation} 
\end{lem}
\begin{proof}
   Let $w_1$ be a real parameter, $P(w_1)=\prod_{p\le w_1}p$ and $\ccP=\{p\leq w_1\}$. Our strategy to prove this result is based on the following observation: Let $b$ be the unique integer less than $P(w_1)$ satisfying the congruences $b\equiv -a_{p}\pmod{p}$ for all primes $p\leq w_1$. Then by definition, 
   \[S_{w_1}=S(\mathcal{A}, \mathcal{P})=\sum_{\substack{b+1\le n\le b+y\\ (n,\ccP)=1}}1,\] where $\mathcal{A}=\{n\in \Z: b+1\le n\le b+y\}$. Now, for any $d\mid P(w_1)$, we have \[|\#\{\mathfrak{a} \in \mathcal{A} : d \mid \mathfrak{a}\}-y/d|\le 1.\] Choosing $g(d)=1$, we obtain that $|r_d|\le 1$. With the choices above, the two conditions in Lemma \ref{lmsieve} hold. Thus, applying Lemma \ref{lmsieve} with $\alpha=1$, $z=w_1$ and $D=w_1^u$, we obtain 
    \begin{equation}\label{fl}
        S_{w_1}= y\,\Theta_{w_1} \Big(1 + \theta \,e^{-u\log u+u\log_2 u+\O(u)} \Big) + \mathcal{O}\bigg( \sum_{d \mid P(w_1), d \leq w_1^{u}} 1 \bigg),
    \end{equation}
    uniformly for $u\ge 6$ and $|\theta|\le 1$. 
    Since $\log y\geq \frac12\log_2x$, with $\log y\geq\log_2x$ when
$\lambda\geq1$, and $\log_2y=\log_3x+O(1)$, the hypotheses imply, for
sufficiently large $x$, that
\[
\lambda_k\leq
\min\left\{\lambda',
\exp\left(\frac1{\sqrt7}\sqrt{\log y\log_2y}\right)\right\}.
\]Let $w_1\defeq \l_k^{25/8}$ and $u\defeq 5h/16$ where \[h\defeq \frac{\log y}{\log \l_k}\ge \sqrt{\frac{7\log y}{\log _2 y}},\] and hence, \be\label{hbound} h \log h\ge \({\sqrt{7}}/{2}+o(1)\)\sqrt{\log y\log_2 y}\ge (7/2+o(1))\log \l_k.\ee We shall appeal to this estimate momentarily. We have $w_1^{u}= y^{125/128}$. Therefore, \begin{equation}\label{fl1}\sum_{d \mid P(w_1), d \leq w_1^{u}} 1 \le y^{125/128}.
    \end{equation}
    Using our choice of $u$ and the estimate \eqref{hbound}, we compute: 
    \[u\log u=(5/16+o(1))h\log h\ge (35/32+o(1))\log \l_k. \] 
    Employing this bound with \eqref{fl} results in the estimate
    \alg{
        S_{w_1}&=y\,\Theta_{w_1}\bigg(1 + \mathcal{O}\Big(e^{-u\log u+u\log_2 u+\O(u)}\Big)\bigg)+ \mathcal{O}\left(y^{125/128}\right) \\&=y\,\Theta_{w_1}\left(1 + \mathcal{O}\left( 1/{\lambda_k^{21/20}}\right)\).\qedhere}         
\end{proof}
\begin{rem}\label{rem:constants}
The constants and parameters chosen in the preceding and subsequent sieve arguments are selected to balance three competing constraints. To successfully bound the error terms, we require parameters $\eta_1, \eta_2, \eta_3 > 0$ such that:
\begin{itemize}
    \item \emph{The Sieve Error:} In Section 4, we transition from the local sieve density $\Theta_{w_1}$ to the global density $1/\log x$. This requires the fundamental lemma error to be strictly smaller than $\l_k^{-1}$, which forces the condition $u\log u\ge (1+\eta_1)\log \l_k$.
    \item \emph{The Azuma Variance:} To achieve an $\exp(-3\lambda_k)$ tail bound in the upcoming Lemmas \ref{f3} and \ref{f4}, the martingale variance requires a lower bound $w_1 \ge \lambda_k^{3+\eta_2}$.
    \item \emph{The Level of Distribution:} For the combinatorial sieve estimates to hold, we require $w_1^u \le y^{1-\eta_3}$.
\end{itemize}
These conditions collectively force the regime boundary constant $\eta$ $\Big($where $\log \lambda_k \le \frac{1}{\sqrt{\eta}}\sqrt{\log y \log_2 y}$$\Big)$ to strictly satisfy \[\eta > 2 \frac{3+\eta_2}{1-\eta_3}(1+\eta_1).\] Taking the limit as our parameters vanish shows that $\eta > 6$ is the limiting barrier. Thus, when $\lambda\geq1$, the limiting constant produced by this approach is $\sqrt6$. We take $\eta=7$ in the proof
and use the constant $\sqrt{15}$ in the statement of Lemma \ref{f1} to
obtain a convenient bound that is uniform also for $\lambda<1$.
\end{rem}

\subsection{Sieving for Medium-Sized Primes (Using Azuma's Inequality)}

\begin{lem}\label{f3}
 Let $w_1\defeq  \lambda_0^{25/8}$ and
$w_2\defeq  y^{4/3} $ where $ \l\le\l'$.
Conditional on $\cF^0_{w_1}$ satisfying $S_{w_1}\gg y/\log w_1$, we have
\[
\P_{w_1,w_2}^{\,0}\bigg(|S_{w_2}-\hat\Theta_{w_1,w_2}\,S_{w_1}|
\ge {\hat\Theta_{w_1,w_2}\,S_{w_1}}/\lambda_0^{17/16}\bigg)\ll \exp(-3\lambda_0).
\] 
\end{lem}

\begin{proof}
   Let $p_0\defeq  w_1$ and let $p_1 < \ldots <p_m$
be the primes in $(w_1,w_2]$. We 
define random variables by
\[
X_{j}\defeq \hat\Theta_{w_1,p_j}^{-1}\,S_{p_j}\qquad(j=0,1,\ldots,m)
.\]
The sequence $X_0, X_1,\ldots,X_m$ satisfies the generalized martingale conditions of Lemma \ref{lem:Azuma-ineq} since
\begin{align}
\E^0_{p_j,p_{j+1}}(X_{j+1}|\cF^0_{p_j})&=\hat\Theta_{w_1,p_{j+1}}^{-1}
\E^0_{p_j,p_{j+1}}(S_{p_{j+1}}|\cF^0_{p_j})
\\&=X_j+\hat\Theta_{w_1,p_{j+1}}^{-1}\left(\frac{\big|\ccS_{p_j}(y) \cap (0\bmod{p_{j+1}}) \big|}{p_{j+1}-1}\right)\\& \le X_{j}+\frac{4y\log p_{j+1}}{p_{j+1}^{2}\log w_1}\qquad\(\hat\Theta_{w_1,p_{j+1}}^{-1}\le \frac{2\log p_{j+1}}{\log w_1}\)
\end{align} for large enough $x$. Note that $X_0=S_{w_1}\gg y/\log w_1$. Furthermore, we have \medskip
\begin{align}
|X_{j+1}-X_j|&=\hat\Theta_{w_1,p_j}^{-1}\,
\big|(1-(p_{j+1}-1)^{-1})^{-1}S_{p_{j+1}}-S_{p_j} \big|
\\&\ll \frac{\log p_j}{\log w_1}
({S_{p_{j+1}}}/{p_{j+1}}+S_{p_j} - S_{p_{j+1}}).
\end{align} 
If $p_{j+1}>y/2$, we have that $|X_{j+1}-X_j|\ll {(\log y)}/{\log w_1},$ since $S_{p_{j+1}}\le y$ and $S_{p_j}=S_{p_{j+1}}+\mathcal{O}(1)$.
 In the case $p_{j+1}\le y/2$, the same upper-bound sieve argument as in Lemma~\ref{lem:sieve-upper} shows that for any value of $\a_{p_{j+1}}\bmod{p_{j+1}}$, we have 
\begin{align}|X_{j+1}-X_j| &\ll \frac{\log p_j}{\log w_1}
\left({S_{p_{j+1}}}/{p_{j+1}}+\big|\ccS_{p_j}(y) \cap (\a_{p_{j+1}}\bmod{p_{j+1}})\big|\right)
\\&\ll \frac{\log p_j}{\log w_1}\left(\frac{y/p_{j+1}}{\min\{\log p_j,\log(y/p_{j+1})\}}\right).\end{align} 

Consequently, 
\alg{\sum_{j=0}^{m-1}|X_{j+1}-X_j|^{2}& \ll \frac{y^{2}}{\log^2 w_1}\sum_{w_1<p\le \sqrt{y}}\frac1{p^2}+\frac{y^2 \log^2 y}{\log^2 w_1}\sum_{\sqrt{y}<p\le y/2}\frac1{p^2 \log^2(y/p)}+y^{4/3}\log^2 y\\& \ll \frac{y^2}{w_1\log^3 w_1}.}

\medskip
Also, we know that 
$\sum_{p>w_1}{\log p}/{p^{2}}\ll 1/{w_1}$.
Thus, if $\alpha>0$ is a sufficiently small absolute constant and $x$ is sufficiently large,
then an application of Lemma \ref{lem:Azuma-ineq} shows that
\begin{align}
\label{eq:Pw3w}
\P_{w_1,w_2}^{\,0}\bigg(|X_m-X_0|\ge{ X_0/\lambda_0^{17/16}}\bigg)
&\le 2\exp\left\{-\frac{\alpha X_0^2\,\lambda_0\log^3 \lambda_0}{y^{2}}\right\}\\&\ll \exp\(-3\lambda_0\),
\end{align} where we use the fact that $t={X_0}/{\lambda^{17/16}_0}\ge {8y}/{(\lambda_0^{25/8}\log \lambda_0)}$ for sufficiently large $x$.

\noindent Note that
\[
\P_{w_1,w_2}^{\,0}\(\big|S_{w_2}-\hat\Theta_{w_1,w_2}\,S_{w_1}\big|\ge \hat\Theta_{w_1,w_2}\,S_{w_1}/\l_0^{17/16}\)
=\P_{w_1,w_2}^{\,0}\(\big|X_m-X_0\big|\ge X_0/\l_0^{17/16}\).
\]
In view of \eqref{eq:Pw3w}, this implies that
\[
\P_{w_1,w_2}^{\,0}\(\big|S_{w_2}-\hat\Theta_{w_1,w_2}\,S_{w_1}\big|\ge \hat\Theta_{w_1,w_2}\,S_{w_1}/\l_0^{17/16}\)
\ll \exp(-3\l_0)
\]
holds.
\end{proof}

\begin{lem}\label{f4}
 Let $w_1\defeq  \lambda_k^{25/8}$ and
$w_2\defeq  y^{4/3} $ where $ \l_k\le \l'$.
Conditional on $\cF_{w_1}$ satisfying $S_{w_1}\gg y/\log w_1$, we have
\[
\P_{w_1,w_2}\bigg(|S_{w_2}-\Theta_{w_1,w_2}\,S_{w_1}|
\ge {\Theta_{w_1,w_2}\,S_{w_1}}/\lambda_k^{17/16}\bigg)\ll \exp(-3\lambda_k).
\] 
\end{lem}

\begin{proof}
    The proof is identical to the previous one, except that the random variables $X_j$ now form a martingale sequence; consequently, we apply Lemma \ref{lem:Azuma}.
\end{proof}

\smallskip
\begin{lem}\label{f5}
 Let $\l=(\log x)^{\beta}$ for some constant $\beta\in (0,1)$. 
Also, let $w_1\in [\l/(\log_2 x)^{2/7},\l]$ and $w_2\defeq  y^{4/3} $.
Conditional on $\cF_{w_1}$ satisfying $S_{w_1}\gg y/\log w_1$, we have
\[
\P_{w_1,w_2}\bigg(|S_{w_2}-\Theta_{w_1,w_2}\,S_{w_1}|
\ge {\Theta_{w_1,w_2}\,S_{w_1}}/(\log_2 x)^{2/7}\bigg)\ll_{\beta} \exp(-3\l).
\] 
\end{lem}

\begin{proof}
   Let $p_0\defeq  w_1$ and let $p_1 < \ldots <p_m$
be the primes in $(w_1,w_2]$. We 
define random variables by
\[
X_{j}\defeq \Theta_{w_1,p_j}^{-1}\,S_{p_j}\qquad(j=0,1,\ldots,m)
.\]

\noindent The sequence $X_0, X_1,\ldots,X_m$ is a martingale since
\[
\E_{p_j,p_{j+1}}(X_{j+1}|\cF_{p_j})=\Theta_{w_1,p_{j+1}}^{-1}
\E_{p_j,p_{j+1}}(S_{p_{j+1}}|\cF_{p_j})
=\Theta_{w_1,p_{j+1}}^{-1}
\(1-p_{j+1}^{-1}\)S_{p_j}=X_j.
\]
Proceeding as before, we have 
\alg{\sum_{j=0}^{m-1}|X_{j+1}-X_j|^{2}\ll \frac{y^2}{w_1\log^3 w_1}.}
Applying Lemma~\ref{lem:Azuma} in conjunction with the bound $X_0\gg y/\log w_1$ yields
\begin{align}
\P_{w_1,w_2}\bigg(|X_m-X_0|\ge{ X_0/(\log_2 x)^{2/7}}\bigg)
&\le 2\exp\left\{-\frac{\alpha X_0^2\,w_1\log^3 w_1}{y^{2}(\log_2 x)^{4/7}}\right\}\\&\ll_{\beta} \exp\(-3\l\),
\end{align}  for sufficiently large $x$. As before, the argument is complete.
\end{proof}

\subsection{Sieving for Large Primes (Using Combinatorial Expansion)} \label{rapid3}The results presented here are slight modifications of those in \cite[Section 6]{ford2025}, adapted to incorporate the restriction $\a_p\neq 0$ for the relevant primes $p$.

\begin{lem}[Sieving for $w_2< p \leq w_3$]\label{lem:w4w5}
Let $w_2\defeq y^{4/3}$, 
$w_3\ge y^{4/3}$, and let 
$\vartheta=y^{-1/10}$.
Conditional on~$\cF^0_{w_2}$ satisfying $S_{w_2}\gg y/\log y$, we have 
\[
\P_{w_2,w_3}^{\,0}\Big(\big|S_{w_3}-\hat\Theta_{w_2,w_3}\,S_{w_2}\big| 
\ge \vartheta S_{w_2} \Big) \le \exp(-0.1\vartheta^{2}S_{w_2})
\]
\end{lem}

\begin{proof}
The proof requires only minor modifications to the argument in~\cite[Lemma 6.1]{ford2025}. First, note that the condition $0 \in \ccS_{w_3}$ presents no difficulties, since for every prime $p>w_2$ we have $p>y$. Secondly, the only difference in the proof is that each occurrence of the factor $p$ in the relevant products is replaced by $p-1$.
\end{proof}

By applying \eqref{31} with $k=0$ and Lemmas \ref{f3} and \ref{lem:w4w5} $\(\text{with } w_3=y^8\)$ together with the estimate \[\hat\Theta_{w_1,w}=\Theta_{w_1,w}\(1+\O\({1}/{\l_0^{3}}\)\)\] for all $w>w_1$, we obtain the following result.

\begin{coro}[Sieving for $w_1 < p \leq y^8$]\label{cor:w1w5}
Let $\l\le \l'$. Suppose that, for every prime $p\le y^8$, the residue class $\a_p \bmod p$ is chosen indpendently and uniformly from $\{1,\ldots,p-1\}$. Then,
with probability $1-\mathcal{O}(e^{-3\l_0})$, we have
\begin{equation}\label{corsieve}
S_{y^8}=y\,\Theta_{y^8}\(1+\mathcal{O}(1/\l_0^{21/20})\).
\end{equation} 
\end{coro}

Next, we require a variant of a result from \cite[Lemma 6.3]{ford2025} to handle sieving by large primes. For completeness, we include the proof here. We also note that the original published proof contained an error, which has been corrected in the arXiv version of \cite{ford2025}; our argument follows the corrected approach. 

\begin{lem}[Sieving for $w_3 < p \leq z$, I]\label{lem:w_3w_4}
Let $x\ge z\ge w_3\ge y^4$ and $\mathcal{P}$ be a set of primes in $(w_3,z]$ such that $\sum_{p\in \mathcal{P}}1/p \ge 1/10$. Let $\ccS\subset \ccS_{w_3}(y)$.
Conditional on $\cF^0_{w_3}$, we have
for all $0\le g\le |\ccS|$:
\[
\P_{\mathcal{P}}^{\,0}\(\Big|\ccS \setminus \bigcup_{p\in \mathcal{P}} (\a_p \bmod p)\Big|=g\)=
(1-\Theta)^{|\ccS|-g}\Theta^g \binom{|\ccS|}{g}(1+\mathcal{O}(y^3/w_3)),
\]
where
\[
\Theta \defeq  \prod_{p\in \ccP} (1-1/p).
\] 
\end{lem}

\begin{proof}
Put $\ell\defeq |\ccS|$,
and assume that $\ell\ge 1$
(the case $\ell= 0$ being trivial). Take $m\defeq \ell-g$, and let
$\cT$, $\mathbf{r}$, $E(\cT,\mathbf{r})$ and $h$ be defined as
in \cite[Lemma 6.3]{ford2025} with $|\cT|=m=\ell-g$. Therefore, as in \cite[Eq. 6.4]{ford2025} (with $p$ replaced by $p-1$), we have
\be\label{eq:Er2}
\P_{\mathcal{P}}^{\,0}(E(\cT,\mathbf{r}))
=\binom{h}{r_1\;r_2\;\cdots\;r_m}
\prod_{p\in \ccP}\bigg(1-\frac{\ell}{p-1}\bigg)
\sum_{\substack{p_1,\ldots,p_h\in\ccP\\p_1<\cdots<p_h}}
\prod_{j=1}^h\frac{1}{p_j-\ell-1}. 
\ee  Let $T_h$ be the sum over $p_1, \dots, p_h$ in \eqref{eq:Er2}. Summing over all vectors $\mathbf{r}$, we find that
$$
\mathbb{P}_{\mathcal{P}}^{\,0}(|\mathcal{S} \setminus \cup_{p \in \mathcal{P}} (\a_p \bmod p)| = \ell - m) = \sum_{\substack{T \subset \mathcal{S} \\ |T| = m}} \sum_{h}\sum_{r_1 + \dots + r_m = h} \binom{h}{r_1 r_2 \dots r_m} V{T_h}
$$
$$
= V \binom{\ell}{m} \sum_{\substack{r_1, \dots, r_m \ge 1 }} \frac{(r_1+\cdots+r_m)!}{r_1! \dots r_m!} T_{r_1+\cdots+r_m}
$$
where
$$
V \defeq  \prod_{p \in \mathcal{P}} \left( 1 - \frac{\ell}{p-1} \right)=\Theta^{\ell}\left(1+\mathcal{O}(y^2/w_3)\right).
$$ 
When $m=0$, the sum on the right side is interpreted to be 1. We have
$$
T_h = \frac{1}{h!} \left( \sum_{p \in \mathcal{P}} \frac{1}{p - \ell-1} +\O \left( \frac{h}{w_3} \right) \right)^h
$$
$$
= \frac{1}{h!} \left( \sum_{p \in \mathcal{P}} \frac{1}{p} + \O \left( \frac{h + \ell}{w_3} \right) \right)^h
$$
$$
= \frac{(-\log\Theta +\O(y^2/w_3))^h}{h!},
$$
provided that $h \le y^2$. For any $h$, we also have the crude upper bound
$$
T_h \le \frac{1}{h!} \left( \sum_{p \in \mathcal{P}} \frac{1}{p - \ell-1} \right)^h \le \frac{(2\log_2 x)^h}{h!}.
$$
Assuming that $m\ge 1$, let
\[
\alpha = \frac{y^2}{m \log_2 x}.
\]
Since $m\le 10y$, we have $\alpha \ge \frac{y}{10\log_2 x}\ge \frac{(\log x)^{1/2}}{10\log_2 x}$.
Thus,
\begin{align*}
\sum_{\substack{r_1,\ldots,r_m\ge 1 \\ h\defeq r_1+\cdots+r_m>y^2}} \frac{h! T_h}{r_1!\cdots r_m!} &\le
\sum_{r_1,\ldots,r_m\ge 0} \frac{(2\log_2 x)^{r_1+\cdots+r_m}}{r_1!\cdots r_m!} \alpha^{r_1+\cdots+r_m-y^2}\\
&=e^{2m\alpha \log_2 x-y^2 \log \alpha}=e^{2y^2-y^2 \log \alpha} < e^{- 2y^2}
\end{align*}
if $x$ is large enough.  It follows that
\begin{align*}
\sum_{\substack{r_1,\ldots,r_m\ge 1 \\ h\defeq r_1+\cdots+r_m}} \frac{h! T_h}{r_1!\cdots r_m!} &= \O(e^{-2y^2})+
\sum_{\substack{r_1,\ldots,r_m\ge 1}} \frac{(-\log\Theta+\O(y^2/w_3))^{r_1+\cdots+r_m}}{r_1!\cdots r_m!} \\
&= \O(e^{-2y^2})+\Big( e^{-\log\Theta+\O(y^2/w_3)} - 1 \Big)^m\\
&= \O(e^{-2y^2})+\big( 1 + \O(y^3/w_3) \big) \big( \Theta^{-1}-1 \big)^m\\
&=\big( 1 + \O(y^3/w_3) \big) \big( \Theta^{-1}-1 \big)^m,
\end{align*}
using in the last step that $(\Theta^{-1}-1)^m\ge 10^{-10 y}$ and $w_3\le x \le e^{y^2}$.
Bringing together all our bounds, we are done.
\end{proof}

\begin{coro}[Sieving for $w_3 < p \leq z$, II]\label{cor:binomSk}
Let $x^{1/2}\ge z^{1/2} \ge w_3 \ge y^4$. Assuming that $0\in \ccS_{w_3}$, we have
\[
\E_{w_3,z}^{\,0} \binom{S_z}{k}=\Theta_{w_3,z}^k\binom{S_{w_3}}{k}(1+\O(y^3/w_3)).
\] 
\end{coro}

\begin{proof}
Let $w\defeq w_3$ and $\Theta\defeq \Theta_{w,z}$.
By Lemma~\ref{lem:w_3w_4} with $\ccS\defeq \ccS_w\cap [1,y]$ and $\ccP$ the set of primes in $(w,z]$, we have
\begin{align*}
\E_{w,z}^{\,0}\binom{S_z}{k}&=(1+\O(y^3/w))\sum_{g=k}^{S_w}
(1-\Theta)^{S_w-g}\Theta^g \binom{S_w}{g}\binom{g}{k}\\
&=(1+\O(y^3/w))\Theta^k\binom{S_w}{k}\sum_{j=0}^{S_w-k}
(1-\Theta)^{S_w-k-j}\Theta^j \binom{S_w-k}{S_w-k-j}\\
&=(1+\O(y^3/w))\Theta^k\binom{S_w}{k}.\qedhere
\end{align*}
\end{proof}

\begin{lem}\label{Ebo}
    Uniformly for $x\ge z\ge \sqrt{x}$ and integer $K\ge 800 \l$, we have \[\E_{z} \binom{S_z}{K}\ll e^{-K}.\]
\end{lem}
\begin{proof}
    By the law of total expectation and an application of \cite[Corollary 6.4]{ford2025}, we find that 
\begin{align*}
\E_{z}\binom{S_{z}}{K} &= \E_{y^8}\left[\E_{z}\(\binom{S_{z}}{K}\,\bigg| \cF_{y^8}\)\right]=\E_{y^8}\(\Theta_{y^8,z}^K\binom{S_{y^8}}{K}(1+\O(y^{-5}))\).
\end{align*} The upper bound sieve (Lemma~\ref{lem:sieve-upper}) implies
the crude bound $S_{y^8} \le 3y/\log y$. Using the standard upper bound $\binom{n}{K}\le (en/K)^{K}$, this implies \[\binom{S_{y^8}}{K}\le \(\frac{3ey}{K\log y}\)^{K}.\] By the inequality $\Theta_{y^{8},z}\le 20(\log y)/\log x$, we deduce \[\E_{z}\binom{S_{z}}{K}\ll \(\frac{60ey}{K\log x}\)^K= \(\frac{60e\l}{K}\)^K<e^{-K}.\qedhere\]
\end{proof}

\begin{lem}\label{ford}
    Let $w_2\defeq y^{4/3}$ where $y\le z^{1/16}$ and $\ccP$ be the set of primes in $(w_2,z]$ such that $x^{1/16}\le z\le x^{3/5}$. Conditional on $\cF_{w_2}$ satisfying $S_{w_2}\gg y/\log y$, we have for all $0\le g\le \sqrt{S_{w_2}}$, \alg{\P_{\mathcal{P}}&\(\Big|\ccS_{w_2}(y) \setminus  \bigcup_{p\in \mathcal{P}} (\a_p \bmod p)\Big| =g\)\\[2ex]&=\frac{\(\Theta S_{w_2}\)^g}{g!}\exp\bigg[-\Theta S_{w_2}+\O\Big(\big(\Theta S_{w_2}+g\big)\big(\Theta+y^{-1/10}\big)+y^{-5}\Big)\bigg]+\O\(e^{-y^{3/5}}\),} where \[\Theta\defeq \prod_{p\in \ccP}\(1-1/{p}\).\]
\end{lem}
\begin{proof}
    Put $v\defeq y^8$ and $\vartheta\defeq y^{-1/10}$. By \cite[Lemma 6.1]{ford2025},
\[
\P_{w_2,v}\( \left|S_v-\Theta_{w_2,v}S_{w_2}\right|\geq \vartheta S_{w_2}\)
\leq \exp\(-0.1\vartheta^2S_{w_2}\)\ll e^{-y^{3/5}},
\]
where we used $S_{w_2}\gg y/\log y$. Let
\[
\mathcal{G}:=\left\{\left|S_v-\Theta_{w_2,v}S_{w_2}\right|<\vartheta S_{w_2}\right\}.
\]
Fix a choice of $\cF_v$ that satisfies $\mathcal{G}$. We may then employ \cite[Lemma 6.3]{ford2025} with $\ccS=\ccS_v(y)$. Since $v=y^8\leq z^{1/2}$,
\[
\sum_{v<p\leq z}\frac{1}{p}
=\log\(\frac{\log z}{\log v}\)+o(1)\ge\log 2+o(1)>\frac{1}{10}.
\]
Moreover, the elements of $\mathcal S_v(y)$ are distinct modulo $p$ for $p>v$. Consequently,
\be\label{BFT1}
\P_{v,z}\( S_z=g\,|\,\cF_v\)
=\(1-\Theta_{v,z}\)^{S_v-g}\Theta_{v,z}^g\binom{S_v}{g}
\(1+\O\( y^{-5}\)\).
\ee
On $\mathscr G$,
\[
S_v=\Theta_{w_2,v}S_{w_2}\(1+\O\( y^{-1/10}\)\).
\] 
Since $\Theta=\Theta_{w_2,v}\Theta_{v,z}$ and $\Theta_{w_2,v}\asymp1$, we obtain
\be\label{BFT2}
\(1-\Theta_{v,z}\)^{S_v}
=\exp\(-\Theta S_{w_2}\(1+\O\(\Theta+y^{-1/10}\)\)\).
\ee
Furthermore, uniformly for $0\leq g\leq\sqrt{S_{w_2}}$,
\be\label{BFT3}
\begin{aligned}
\(1-\Theta_{v,z}\)^{-g}\Theta_{v,z}^g\binom{S_v}{g}
&=\frac{1}{g!}\(\frac{\Theta_{v,z}S_v}{1-\Theta_{v,z}}\)^g
\prod_{0\leq j<g}\(1-\frac{j}{S_v}\)\\
&=\frac{1}{g!}\(\Theta S_{w_2}\)^g
\(1+\O\(\Theta+y^{-1/10}\)\)^g,
\end{aligned}
\ee
where $g/S_v\ll S_{w_2}^{-1/2}\ll y^{-1/10}$. 
Combining \eqref{BFT1}, \eqref{BFT2}, and \eqref{BFT3}, we get that
\alg{
\P_{v,z}\( S_z=g\,|\,\cF_v\)
=\frac1{g!}&\exp\Big[-\Theta S_{w_2}\big(1+\O(\Theta+y^{-1/10})\big)\Big](\Theta S_{w_2})^g\\&\times\(1+\O(\Theta+y^{-1/10})\)^{g}\(1+\O(y^{-5})\),
}
uniformly over all choices of $\cF_v$ for which $\mathcal{G}$ holds. Therefore, 
\alg{
\P_{\ccP}\( S_z=g\,|\,\mathcal{G}\)&=\frac1{g!}\exp\Big[-\Theta S_{w_2}\big(1+\O(\Theta+y^{-1/10})\big)\Big](\Theta S_{w_2})^g\\&\times\(1+\O(\Theta+y^{-1/10})\)^{g}\(1+\O(y^{-5})\).
}
Finally, using
\[
\P_{w_2,v}\(\mathcal G^c\)\ll e^{-y^{3/5}},
\]
we may conclude that 
\alg{
\P_{\ccP}\(S_z=g\)&=\P_{w_2,z}\(S_z=g\,|\,\mathcal{G}\)\P_{w_2,v}\(\mathcal{G}\)+\O\(e^{-y^{3/5}}\)
\\&=\frac1{g!}\exp\Big[-\Theta S_{w_2}\big(1+\O(\Theta+y^{-1/10})\big)\Big](\Theta S_{w_2})^g\\&\times\(1+\O(\Theta+y^{-1/10})\)^{g}\(1+\O(y^{-5})\)+\O\(e^{-y^{3/5}}\)\\&
=\frac{\(\Theta S_{w_2}\)^g}{g!}\exp\bigg[-\Theta S_{w_2}+\O\Big(\big(\Theta S_{w_2}+g\big)\big(\Theta+y^{-1/10}\big)+y^{-5}\Big)\bigg]+\O\(e^{-y^{3/5}}\).
}
\end{proof}
\section{Endgame: Part-1 (Small $\l$)}\label{section1}

Our strategy in this section and Section \ref{section2} relies on a probabilistic interpretation of the gap counts. We first invoke Lemma \ref{probru} to relate the quantities $M_{\cA}$ and $N_{\cA}$ to the distribution of the random variables $S_{z(t)}$. We then deploy the random sieve estimates from Section \ref{sieve} to evaluate the corresponding probabilities.  We retain some of the assumptions from Section \ref{sieve}. That is, for sufficiently large $x$, we set $y\defeq \l \log x$, where $0<\l\le \log x$.  Recall from \eqref{lk} and \eqref{l0} that \[\lambda_k\defeq \max\{\lambda,k, k\log (k/\l), (\log_2 x)^{200}\},\] and \[\lambda_0\defeq \max\{\lambda, (\log_2 x)^{200}\}.\]  Here, we follow the aforementioned approach to prove Theorems \ref{thm:HLgaps} and \ref{thm:HLpoisson}. To begin, we assume for this section that $\cA$ satisfies $\mathbf{HL}\left[y,\exp\(\sqrt{\log_2 x\log_3 x}\)\right]$. 

\smallskip
\noindent We assume throughout this section that $\l_k\le \l'$, which implies the upper bound  $y\le \log^2 x$ for $x>e^e$. For any integer $k\in [0,\l']$, we define an integer parameter \be\label{LK}\mathcal{K}_{k}\defeq k^2+\flr{800\l_k}.\ee  We also define the parameter $x_k$ as
\be\label{xk}x_k\defeq x^{1-1/\l_k^{21/20}},\ee where the exponent $21/20$ arises from the sieving error discussed in Remark \ref{rem:constants}. If $\l\ge 1/\sqrt{\log x}$, then  \be\label{LK1}\cK_k\le \min\{(\l')^2+800\l',y/3\}\le \min\Big\{\frac1{6}\exp\(\sqrt{\log_2 x\log_3 x}\),y/3\Big\},\ee and $x_k\ge y^{100}$ for sufficiently large $x$ and any non-negative integer $k\le \l'$. The choices above also yield \be x_k y^{2\cK_k+2}\ll xe^{-\cK_k \log_2 x},\ee uniformly in the present range. so the final error term in Lemma \ref{probru} may be absorbed into the existing error terms below. We define the truncated counting functions \[M'_{\cA}(x;y)\defeq \big|\{ x_0<n\le x: n\in \cA\text{ and } [n+1,n+y]\cap \cA=\varnothing\}\big|,\] and \[N'_{\cA}(x;y,k)\defeq \big|\{x_k<n\le  x:|[n+1,n+y]\cap\cA|=k\}\big|.\] 

\medskip
\noindent By \eqref{r} and \eqref{Gall}, we trivially have the relations \[M_{\cA}(x;y)=M'_{\cA}(x;y)+\O\(x^{1-1/\l_0^{21/20}}\),\] and \[N_{\cA}(x;y,k)=N'_{\cA}(x;y,k)+\O\(x^{1-1/\l_k^{21/20}}\).\] To verify that the error terms are negligible, it will suffice to prove that \[ \l_k^2\,e^{\l}\,k!\,\l^{-k}\,\log x< x^{1/\l_k^{21/20}}.\]Using the crude bound $k!\le k^k$ and taking logarithms on both sides, it is clear (since $\l_k=(\log x)^{o(1)}$) that \[\log x\ge 4\l_k^{3}\ge \l_k^{21/20}(\l+2\log \l_k+k\log (k/\l)+\log_2 x).\]

\medskip
\subsection{Proof of Theorem \ref{thm:HLgaps}}\label{exp} We split the proof into two parts depending on the size of $\l$. 

\subsubsection{Case 1: $\l\ge 1/\sqrt{\log x}$}
As stated at the beginning of this section, we start by applying Lemma \ref{probru} (ii) with $x'=x_0$ and $K=\cK_0$ as above. Thus, we can write $M'_{\cA}(x;y)$ as \be\label{UK1}
M'_{\cA}(x;y)= \int_{x_0}^{x}\P\(0\in \ccS_{z(t)}\)\( \P_{z(t)}^{\,0}(S_{z(t)}=0) +\O\(
\E_{z(t)}^{\,0}\binom{S_{z(t)}}{\cK_0+1}\)\)\, \dd t+\O\(xe^{-\cK_0\log_2 x}\).
\ee 

We handle the expression in (\ref{UK1}) term by term. Let $w=y^8$, $z_1\defeq z(x_0)$ and $z_2\defeq z(x)$. Firstly, it is easy to observe that \be\P\(0\in \ccS_{z(t)}\)= \prod_{p\le z(t)}\(1-\frac1{p}\)\in \left[\Theta_{z_2},\Theta_{z_1}\right].\ee By Mertens' theorem, we find that \be\label{euler}\Theta_{z_1}=\frac1{\log x}\(1+\O\pfrac{1}{\l_0^{21/20}}\)\quad\text{and}\quad\Theta_{z_2}=\frac1{\log x}\(1+\O\pfrac{1}{x^{1/e^{\gamma}}}\),\ee where the error term in the first expression arises from approximating $1/\log x_0$ with $1/\log x$. Hence, we obtain \be\label{eupro}\P\(0\in \ccS_{z(t)}\)=\Theta_{z(t)}=\frac1{\log x}\(1+\O\pfrac{1}{\l_0^{21/20}}\)\quad\text{uniformly for }t\in [x_0,x].\ee 

\medskip
\noindent To bound the expectation appearing in the error term of \eqref{UK1}, we follow an identical argument to that of Lemma \ref{Ebo}. The upper bound sieve (Lemma~\ref{lem:sieve-upper}) implies
the crude bound $S_w \le 3y/\log y$. Corollary \ref{cor:binomSk}
and the bound $\Theta_{w,z_1} \le 20\,(\log y)/{\log x}$
imply that
\begin{align}\label{Ebound}
\E_{z(t)}^{\,0}\binom{S_{z(t)}}{\cK_0+1} &\le \E_{z_1}^{\,0}\binom{S_{z_1}}{\cK_0+1}\notag\\
&\ll \Theta_{w,z_1}^{\cK_0+1} \E_w^{\,0} \binom{S_w}{\cK_0+1}\notag\\
&\ll \(\Theta_{w,z_1} \frac{3ey}{\cK_0\log y}\)^{\cK_0+1} \notag\\
&\ll e^{-\cK_0} \ll e^{-4\l_0},
\end{align}
where we used \eqref{LK} in the last step.
Finally, we need to obtain an asymptotic for  $\P_{z(t)}^{\,0}(S_{z(t)}=0)$. To this end, we deduce the following pair of inequalities:
\be\label{fi}\P_{z_2}^{\,0}\(S_{z_2}=0\)\ge \P_{z(t)}^{\,0}\(S_{z(t)}=0\)\ge \P_{z_1}^{\,0}\(S_{z_1}=0\).\ee

\noindent We now proceed to the {calculation of $\P_{z_j}^{\,0}\(S_{z_j}=0\)$ for $j\in \{1,2\}.$} Let $\mathscr{C}$ be the event that the estimate in \eqref{corsieve} holds for $S_{w}$. By Corollary \ref{cor:w1w5}, we know that $\P^{0}_{z_j}(\mathscr{C}^{c})\ll e^{-3\l_0}$ holds uniformly for $j\in \{1,2\}$. Observe that 
\begin{align*}\P^{\,0}_{z_j}(S_{z_j}=0)&=\PR^{\,0}_{z_j}(S_{z_j}=0 \mid \mathscr{C})\cdot\P^{\,0}_{z_j}(\mathscr{C})+\P^{\,0}_{z_j}(S_{z_j}=0 \mid \mathscr{C}^{c})\cdot\P^{\,0}_{z_j}(\mathscr{C}^{c})\\&=\P^{\,0}_{z_j}(S_{z_j}=0 \mid \mathscr{C})+\O(e^{-3\l_0}).\end{align*} We take $\ccS=\ccS_{w}(y)$ and $\ccP$ as the set of primes in $(w,z_j]$. Thus, applying Lemma \ref{lem:w_3w_4}, we obtain 
\begin{align*} \P^{\,0}_{z_j}(S_{z_j}=0 \mid \mathscr{C})&=\exp\bigg[{y\cdot \Theta_{\,w}\Big(-\Theta_{w,z_j}+\O\(\Theta_{w,z_j}^{2}\)\Big)}\Big(1+\O\(1/\l_0^{21/20}\)\Big)\bigg]\\&=\exp\bigg[{-y\cdot \Theta_{\,z_j}\Big(1+\O\(1/\l_0^{21/20}\)\Big)\Big(1+\mathcal{O}\big((\log y)/\log x\big)\Big)}\bigg]\\&=\exp(-\lambda+\O\(\l/ \l_0^{21/20}\))= e^{-\lambda}\(1+\mathcal{O}\({\l}/{\l_0^{21/20}}\)\),\,\end{align*} where we used Mertens' theorem to bound the relative error $\Theta_{w,z_j}\ll (\log y)/\log x$, and the final equality follows from \eqref{euler} alongside the restriction $\l\le \l_0$.

Combining both these estimates with (\ref{fi}), we arrive at \be\label{fi2}\P^{\,0}_{z(t)}(S_{z(t)}=0)=e^{-\lambda}\(1+\mathcal{O}\({\l}/{\l_0^{21/20}}\)\)\quad\text{ for }t\in [x_0,x]. \ee Collecting estimates \eqref{eupro}, \eqref{Ebound}, \eqref{fi2} and substituting into \eqref{UK1},  we obtain that \[M'_{\cA}(x;y)=\frac{xe^{-\lambda}}{\log x}\(1+\mathcal{O}\pfrac{\l+1}{\l_0^{21/20}}\).\] 

\subsubsection{Case 2: $\l<1/\sqrt{\log x}$} In this bounded regime, the random sieving machinery developed in Section \ref{sieve} is not necessary. Instead, one can rely on a direct argument. We claim that \be \1{rrr}M_{\cA}(x;y)=\frac{x}{\log x}\(1+\O\pfrac{1}{\sqrt{\log x}}\).\ee This matches the asymptotic $xe^{-\l}/\log x$ since  $\l< 1/\sqrt{\log x}$ implies $e^{-\l}=1+\O(1/\sqrt{\log x})$. To prove the claim, we recall from \eqref{r} that \be\label{r1}M_{\cA}(x;y)=\big|\{n\le x: n\in \cA\}\big|-\big|\{n\le x: n\in \cA\text{ and }[n+1,n+y]\cap \cA\neq\varnothing\}\big|.\ee 
For the first term in \eqref{r1}, we can use \eqref{eq:HLA} with $\cH=\{0\}$ to obtain \be\label{r2}\big|\{n\le x: n\in \cA\}\big|=\frac{x}{\log x}+\O\(\frac{x}{\log^2 x}\).\ee
For the second term in \eqref{r1}, we can apply a union bound over all possible gap distances up to $y$. We have: \be\label{r3}\big|\{n\le x: n\in \cA\text{ and }[n+1,n+y]\cap \cA\neq\varnothing\}\big|\le \sum_{k\le y}\big|\{n\le x: n\in \cA\text{ and }n+k\in \cA\}\big|\ee 
To compute the sum, we again use \eqref{eq:HLA} with $\cH=\{0,k\}$ for $k\le y$. Consequently, we have \alg{\sum_{k\le y}\big|\{n\le x: n\in \cA\text{ and }n+k\in \cA\}\big|&= \sum_{k\le y}\(\int_{2}^{x}\frac{\fS(\{0,k\})}{\log^2 t}\dd t+\O\(\frac{x}{\log^3 x}\)\)\\&\ll\frac{x}{\log^2 x}\sum_{k\le \sqrt{\log x}}\fS(\{0,k\})+\O\(\frac{x}{\log^{5/2} x}\).} It is well-known that $\fS(\{0,k\})=0$ for odd $k$. For even $k$, \[\fS(\{0,k\})=\fS_2\prod_{\substack{p>2\\p\mid k}}\(\frac{p-1}{p-2}\),\] where $\fS_2$ is the twin-prime constant \[\fS_2\defeq 2\prod_{p>2}\(1-\frac1{(p-1)^{2}}\)=1.3203236\ldots\] Applying the estimate \cite[Eq. 12]{HLM}, we see that \[\sum_{k\le w}\fS(\{0,k\})={w}+\O(\log w).\] Thus, we get that \[\sum_{k\le \sqrt{\log x}}\big|\{n\le x: n\in \cA\text{ and }n+k\in \cA\}\big|\ll \frac{x}{\log^{3/2} x}.\] Combining this estimate with \eqref{r1}, \eqref{r2}, and \eqref{r3}, it is immediate that \eqref{rrr} holds.

\smallskip
\subsection{Proof of Theorem \ref{thm:HLpoisson}} Under the hypotheses of Theorem \ref{thm:HLpoisson}, we have 
\[\max\{k,k\log(k/\l)\}\le \min\left\{\l',\exp\(\frac1{\sqrt{7}}\sqrt{\log (\l \log x) \log_2 (\l \log x)}\)\right\}.\]
Turning our attention to $N'_{\cA}$, we apply Lemma \ref{probru} (i) with $x'=x_k$ and $K=\cK_k$ as in \eqref{LK} and \eqref{xk}. This implies that $N'_{\cA}(x;y,k)$ can be expressed as \be\label{UK12}
N'_{\cA}(x;y,k)= \int_{x_k}^{x}\P_{z(t)}(S_{z(t)}=k) +\O\((\cK_k+1)^k\,
\E_{z(t)}\binom{S_{z(t)}}{\cK_k+1}\)\, \dd t +\O\(xe^{-\cK_k\log_2 x}\).
\ee 

For $t\in [x_k,x]$, we apply Lemma \ref{Ebo} with $z=z(t)$ and $K=\cK_k+1$. Therefore, we conclude that 
\be \label{E2}(\cK_k+1)^k\,
\E_{z(t)}\binom{S_{z(t)}}{\cK_k+1}\ll (\cK_k+1)^ke^{-\cK_k}\ll e^{-3\l_k}. \ee Just as in \eqref{fi}, we are left to obtain an asymptotic for $\P_{z(t)}(S_{z(t)}=k)$ uniform for $t\in [x_k,x]$. Fix $t\in [x_k,x]$. By the same estimates used to derive \eqref{euler} and \eqref{eupro}, our current choice of $x_k$ yields \be\label{eupro1}\Theta_{z(t)}=\frac1{\log x}\(1+\O\pfrac1{\l_k^{21/20}}\).\ee Combining the estimate \eqref{31} with Lemma \ref{f4} results in the asymptotic \be\label{eq:asymp} S_{y^{4/3}}=y\,\Theta_{y^{4/3}}\(1+\O\(1/\l_k^{21/20}\)\),\ee with probability $1-\O\(e^{-3\l_k}\)$. 

\smallskip
   It follows that $k\le \sqrt{S_{y^{4/3}}}$ for large $x$, since $k\le \l_k=y^{o(1)}$ and $S_{y^{4/3}}\gg y/\log y$. An application of Lemma \ref{ford} ($z=z(t),\, g=k,\,\Theta=\Theta_{y^{4/3},z(t)}$), along with \eqref{eupro1} and \eqref{eq:asymp}, gives 
   \alg{
   \P_{z(t)}(S_{z(t)}=k)=\frac{\l^k}{k!}&\exp\left[-\l+\O\(\frac{\l+k}{\l_k^{21/20}}+(\l+k)\(\Theta+y^{-1/10}\)+y^{-5}\)\right]\\&+\O\(e^{-3\l_k}\),
   } where we used $e^{-y^{3/5}}\ll e^{-3\l_k}$. We may note that $\Theta\ll y^{-1/10}$ by Mertens' theorem. Turning to the error term inside the exponential, we have 
   \[\frac{\l+k}{\l_k^{21/20}}+(\l+k)\(\Theta+y^{-1/10}\)+y^{-5}\ll \l_k^{-1/20},
   \] since $\l+k\le 2\l_k$ and $\l_k=y^{o(1)}$. Consequently, \[\P_{z(t)}(S_{z(t)}=k)=\frac{e^{-\l}\l^k}{k!}\exp\(\O\(\l_k^{-1/20}\)\)+\O\(e^{-3\l_k}\).\] For $k\ge 1$, the bound $k!\le k^k$ implies 
   \[\frac{e^{-\l}\l^k}{k!}\ge \exp\(-\l-k\log (k/\l)\)\ge e^{-2\l_k},\] and the same bound holds for $k=0$ as well. We finally conclude that 
   \[\P_{z(t)}(S_{z(t)}=k)=\frac{e^{-\l}\l^k}{k!}\(1+\O\(\l_k^{-1/20}\)\),\] uniformly for $t\in [x_k,x]$.  Inserting this expression into \eqref{UK12} and integrating over $t\in [x_k,x]$, the estimate \eqref{E2} then implies that \[N'_{\cA}(x;y,k)=\frac{xe^{-\lambda}\l^{k}}{k!}\(1+\mathcal{O}\pfrac{1}{\l_k^{1/20}}\).\] This completes the proof of Theorem \ref{thm:HLpoisson}.\qed

\section{Endgame: Part-2 (Large $\l$)}\label{section2}
In this section, we prove Theorems~\ref{UK4} and  \ref{thm:HLasymp} using a combination of techniques from both the previous section and~\cite[Section 8]{ford2025}. To establish the necessary bounds, we first set up our premise for the proofs. Let $z_2\defeq z(x)$. For any $w<x$ and $h>1$, we set \be N_{\cA}(x,w;h)\defeq \big|\{w<n\le  x:[n+1,n+h]\cap \cA=\varnothing\}\big|.\ee Furthermore, for any $h'>h$, we split the quantity $N_{\cA}(x,w;h)$ into two parts as follows: \be\label{decomp} N_{\cA}(x,w;h)=N_{\cA}(x,w;h,h')+ N_{\cA}(x,w;h'),\ee where \[N_{\cA}(x,w;h,h')\defeq |\{w< n\le x: [n+1,n+h]\cap \cA=\varnothing\textrm{ and }[n+1,n+h']\cap \cA\neq\varnothing\}|.\]

\smallskip
First, we give a general upper bound for $N_{\cA}(x,w;h)$, which will be useful for proofs of the subsequent theorems. 

\begin{lem}\label{crux}
    Let $\v_0,\v_1\in (0,1)$ with $\v_0>\v_1$. Assume that \be\label{kappa} \cA\subset \N\text{ satisfies }{\mathbf{HL}}\left[\log^{1+\v_0} x,\log^{\v_0} x\right].\ee Furthermore, let $w=w(x)$ be any parameter satisfying $\log w \sim \log x$ and $w\le x\exp(-2\log^{\v_0} x)$ for sufficiently large $x$. We have \be \label{decompupper}N_{\cA}(x,w;\log^{1+\v_1} x)\le x\exp\(-\(1-o_{\v_0,\v_1}(1)\)f\(1+\frac1{\v_1}\)\log^{\v_1} x\).\ee
\end{lem}

\begin{proof}
    We set $y_1\defeq \log^{1+\v_1}x$ for simplicity. We begin by applying Lemma \ref{probru} (i) with $y_1$ in place of $y$, lower cutoff $x'=w$, $k=0$, and $K=\flr{800\,\log^{\v_1}x}$. Under this substitution, the random variable $S_{z(t)}$ counts the elements in $[1,y_1]$. Verifying the hypotheses of Lemma \ref{probru}, we see that $x'\ge (y_1)^{100}$ and \[K\le \min\Big\{(1/6)\log^{\v_0} x,y_1/3\Big\}\]for sufficiently large $x$ depending only on $\v_0$ and $\v_1$. Therefore, we obtain \be \label{U1} N_{\cA}(x,x';y_1)=\int_{x'}^{x}\P_{z(t)}\(S_{z(t)}=0\)\,\dd t+ \O\(\int_{x'}^{x}\,\E_{z(t)}\binom{S_{z(t)}}{K+1}\,\dd t\)\,+\O\(xe^{-K\log_2 x}\).\ee  With our setup in place, we can estimate the quantities in (\ref{U1}). Arguing as in Lemma \ref{Ebo}, by \cite[Corollary 6.4]{ford2025} and the crude bounds $S_{y_1^8}\le 3y_1/\log y_1$, $\Theta_{y_1^8,z_1}\le 20\,{(\log y_1)}/{\log x}$, we deduce \be\label{expect}\E_{z(t)}\binom{S_{z(t)}}{K+1}\ll e^{-K}\ll \exp\(-4\log^{\v_1} x\).\ee 
    It remains to bound $\P_{z(t)}\(S_{z(t)}=0\)$ from above. It suffices to prove an upper bound for $\P_{z_2}\(S_{z_2}=0\)$ since $\P_{z_2}\(S_{z_2}=0\)\ge \P_{z(t)}\(S_{z(t)}=0\)$ for $t\in [x',x]$. 
 Applying Lemma \ref{lem:sieve-lower} ($y=y_1$ and $z=\log^{\v_1} x$), we find that \[S_{\log^{\v_1} x}\ge (1/\v_1-o(1))f\(1+1/\v_1\)e^{-\gamma}y_1/\log_2 x.\]

\smallskip
 Combining this estimate with Lemma \ref{f5} ($w_1=\log^{\v_1} x$ and $\beta=\v_1$) yields the lower bound: \be\label{lower1}S_{y_1^{4/3}}\ge (3/4-o(1))f\(1+1/\v_1\)e^{-\gamma}y_1/\log y_1\ee with probability $1-\O\(\exp\(-3\log^{\v_1} x\)\)$. This implies that \alg{\P_{z_2}\(S_{z_2}=0\)&\le \P_{z_2}\(S_{z_2}=0\mid \eqref{lower1}\text{ holds}\)+\P\(\eqref{lower1}\text{ fails}\)\\&\le\P_{z_2}\(S_{z_2}=0\mid \eqref{lower1}\text{ holds}\)+\O\(\exp\(-3\log^{\v_1} x\)\).}Moreover, an application of Lemma \ref{ford} ($z=z_2,\,g=0,\,\Theta=\Theta_{y_1^{4/3},z_2}$) tells us that \alg{\P_{z_2}\(S_{z_2}=0\)&\le \exp\bigg[(1 - o(1))\, f\(1 + \frac{1}{\v_1}\)
    \, e^{-\gamma} \frac{3\,y_1}{4\log y_1}\cdot \(-\Theta_{y_1^{4/3},z_2}+\O\(\Theta_{y_1^{4/3},z_2}y_1^{-1/10}\)\)\bigg]\\&\le \exp\(-\(1-o(1)\)f\(1+\frac1{\v_1}\)\log^{\v_1} x\).} Employing this estimate along with \eqref{U1} and \eqref{expect} leads to the bound \[N_{\cA}(x,x';y_1)\le x\exp\(-\(1-o(1)\)f\(1+\frac1{\v_1}\)\log^{\v_1} x\).\qedhere\] \end{proof}

\smallskip
Next, we develop the general setup for the lower bounds.
\begin{lem}\label{lo}
    We have $M_{\cA}(x;h)\ge N_{\cA}(x,w;h,h')/h'$ for $w\ge \min \cA$.
\end{lem}
\begin{proof}
    Let $\cA=(a_n)$.
 We begin by observing that $N_{\cA}(x,w;h,h')$ is the number of $w< n\le x$ for which the interval $[n+1,n+\lfloor h\rfloor ]$ contains no elements from $\cA$ while the interval $[n+1,n+\lfloor h'\rfloor]$ contains at least one such element. The condition on $n$ holds if and only if there is a $j$ for which $a_j\le n$ and $a_{j+1}\in [ n+\flr{h}+1,n+\flr{h'}] $, which can occur only if $a_{j+1}-a_j\ge \flr{ h}+1$. Hence, in this case, we have \[\max(a_j,a_{j+1}-\lfloor h'\rfloor)\le n \le a_{j+1}-\flr{h}-1,\] and there are at most $\lfloor h'\rfloor-\flr{h}$ such $n$'s for this $a_j$. Therefore, \alg{N_{\cA}(x,w;h,h')&\le h'\sum_{\substack{ a_j\le x\\ a_{j+1}-a_j\ge \lfloor h\rfloor+1}}1\le h'\,M_{\cA}(x;h).} We arrive at the lower bound \[ M_{\cA}(x;h)\ge {N_{\cA}(x,w;h,h')}/{h'}.\qedhere \] 
\end{proof}

\medskip
With our general bounds established, we now fix the parameters for the proofs of our main theorems. As in Section \ref{section1}, we write $y\defeq \l \log x$, noting that $y\ll \log^2 x$. We assume for the remainder of this section that $\l\ge (\log_2 x)^{200}$ for sufficiently large $x$. Recalling \eqref{LK}, this implies that $\cK_0=\flr{800\l}$. For sufficiently large $x$, we introduce the  lower cutoff $x'\ge y^{100}$, and we choose $y'$ such that $y'>y$.

\smallskip
Dropping the parameter $k$ from \eqref{Gall}, we write $N_{\cA}(x;y)\defeq N_{\cA}(x;y,0)$. Comparing this to  the quantity $N_{\cA}(x,w;h)$ with $w=x'$ and $h=y$, we trivially obtain 
\be\label{imp} N_{\cA}(x;y)=N_{\cA}(x,x';y)+\O\(x'\).\ee We are now in a position to prove Theorems \ref{UK4} and \ref{thm:HLasymp}.
\smallskip
\subsection{Proof of Theorem \ref{UK4}}\label{rapid1}
For the following proof, we assume $c,c'>0$ are fixed parameters as given in Theorem \ref{UK4}. Consequently, all implied constants are permitted to depend on $c$ and $c'$. Assume that \[\cA\text{ satisfies }\mathbf{HL}\left[\log^{1+c'} x,\log^{c'} x\right].\] Let $\l\defeq \log^c x$ and choose $\v_2\defeq(c+c')/2$. We set $x'\defeq  x\exp\(-2\log^{c'} x\)$ and $y'\defeq \log^{1+\v_2} x$. Put $z_1\defeq z(x')$.
\subsubsection{Upper bound for $N_{\cA}$ and $M_{\cA}$:}  Taking $\v_0=c'$ and $\v_1=c$ in Lemma \ref{crux}, it follows from \eqref{imp} that \[M_{\cA}(x;y)\le N_{\cA}(x;y)=  N_{\cA}(x,x';y)+\O(x')\le x\exp\(-\(1-o(1)\)f\(1+\frac1{c}\)\l\).\] This proves the upper bound in Theorem \ref{UK4}.\qed
 
 \subsubsection{Lower bound for $N_{\cA}$ and $M_{\cA}$:}\label{lower2} We deduce from Lemma \ref{lo} and \eqref{decomp} with $w=x'$, $h=y$, and $h'=y'$ that \begin{align*} M_{\cA}(x;y) \ge\frac{N_{\cA}(x,x';y)-N_{\cA}(x,x';y')}{\log ^{1+\v_2} x}.\end{align*} Therefore, we get that \[M_{\cA}(x;y)\ge \frac{N_{\cA}(x,x';y)}{\log^{1+\v_2} x}+\O\bigg(x\exp\(-\(1-o(1)\)f\(1+\frac1{\v_2}\)\log^{\v_2} x\)\bigg),\] which again follows from Lemma \ref{crux} with $\v_0=c',\v_1=\v_2$ and $w=x'$. Thus, it suffices to show a lower bound for $N_{\cA}(x,x';y)$. By another application of Lemma \ref{probru} (i) with $k=0$ and $K=\cK_0$ (we can check that $\cK_0\le \frac1{6}\log^{c'} x$ for large $x$) , we find that \be \label{U11.1} N_{\cA}(x,x';y)=\int_{x'}^{x}\P_{z(t)}\(S_{z(t)}=0\)\,\dd t+ \O\(\int_{x'}^{x}\,\E_{z(t)}\binom{S_{z(t)}}{\cK_0+1}\,\dd t\)\,+\O\(xe^{-\cK_0\log_2 x}\),\ee where the additional term $x'y^{2K_0+2}$ from Lemma \ref{probru} is absorbed into the first error term.  As before, using Lemma \ref{Ebo} with $z=z(t)$ and $K=\cK_0+1$ yields\be\label{U7}\E_{z(t)}\binom{S_{z(t)}}{\cK_0+1}\ll e^{-4\l}.\ee  It remains to bound $\P_{z_1}\(S_{z_1}=0\)$ since $\P_{z_1}\(S_{z_1}=0\)\le \P_{z(t)}\(S_{z(t)}=0\)$ for $t\in[x',x]$. Let $\v_3>0$ be very small and $ w\defeq  \l/{(\log_2 x)^{2/7}}$. For brevity, we denote 
 \[
 L_{\v_3}\defeq\frac{c}{c-\v_3}\(f^{+}\left(\frac{c + 1}{c - \v_3}\right) +\v_3\).
 \] 
 By the definition of $f^{+}$, 
we find that there exists a choice of residues $\a=(\a_p\bmod{p}: p \le w)$ such that
\[
S_w \le S_{(\log x)^{c-\v_3}}
    \le 
        L_{\v_3}
      e^{-\gamma} \frac{y}{\log w}
\]
for all sufficiently large $x$ depending on $\v_3$. Fix $\cF_w=\a$ so that $S_w$ is fixed. Its probability is exactly 
\be\label{ir1}
\P_{w}(\cF_w=\a)= \prod_{p\le w}\frac1{p}=e^{-o(\lambda)}.
\ee
 It is clear that \be\label{cond}\P_{z_1}\(S_{z_1}=0\)\ge \P_{z_1}\(S_{z_1}=0\,|\, \cF_w=\a\) \cdot \P_{w}\(\cF_w=\a\).\ee 
 By Lemma \ref{lem:sieve-lower} ($z=w$), we know that $S_{w}\gg_c y/\log w$. Employing Lemma \ref{f5} ($w_1=w$ and $\beta=c$), we obtain the bound \[\frac{y}{\log y}\ll_{c}S_{y^{4/3}}\le \, 
    \(L_{\v_3}+o_{\v_3}(1)\)
    \frac{3e^{-\gamma}}{4} \frac{y}{\log y}
\] with probability $1-\O\(e^{-3\lambda}\)$. Invoking Lemma \ref{ford} ($z=z_1,\,g=0,\,\Theta=\Theta_{y^{4/3},z_1}$) yields \alg{\P_{z_1}\(S_{z_1}=0\mid \cF_w=\a\)&\ge \exp\bigg[\(L_{\v_3}+o_{\v_3}(1)\)
    e^{-\gamma} \frac{3y}{4\log y}\cdot \(-\Theta_{y^{4/3},z_1}+\O\(\Theta_{y^{4/3},z_1}y^{-1/10}\)\)\bigg]\\&\ge \exp\(-\(L_{\v_3}+o_{\v_3}(1)\)\lambda\).} Multiplying this probability by the one in \eqref{ir1}, it is immediate from \eqref{cond} that \[\P_{z_1}\(S_{z_1}=0\)\ge \exp\(-\(L_{\v_3}+o_{\v_3}(1)\)\lambda\).\]Together with the estimates in \eqref{U11.1} and~(\ref{U7}), we infer that \[N_{\cA}(x,x';y)\ge x\exp\(-\(L_{\v_3}+o_{\v_3}(1)\)\l\).\]  This, in turn, implies that \[M_{\cA}(x;y)\ge N_{\cA}(x,x';y)/(2\log^{1+\v_2} x)\ge x\exp\(-\(L_{\v_3}+o(1)\)\l\),\] as $x\rightarrow\infty$. Since \[\lim_{{\v_3\rightarrow0+}}L_{\v_3}=f^{+}\big((1+1/c)+\big),\] taking the limit as $\v_3\rightarrow 0^{+}$ yields 
    \[
    M_{\cA}(x;y)\ge x\exp\left[-\bigg(f^{+}\Big(\big(1+1/{c}\big)+\Big)+o(1)\bigg)\l\right],
    \] finishing the proof.\qed

\subsection{Proof of Theorem \ref{thm:HLasymp}}\label{rapid}  We fix $c>0$. Henceforth, all implied constants depend at most on $c$. Assume that \[\cA\text{ satisfies }\mathbf{HL}\left[\log^{1+c} x,\log^{c} x\right].\] Let  $\l $ satisfy  \[\l\ge \l''\text{ with }\l=(\log x)^{o(1)}.\] Denote $x'\defeq x^{1-1/\l^{21/20}}$ and $y'\defeq \log ^{1+c/2} x$. Let $z_1\defeq z(x')$. The proof proceeds similarly to that of Theorem \ref{UK4}. 

The primary quantity of interest is $N_{\cA}(x,x';y)$. Observing that $\cK_0\le \frac1{6}\log^c x$ for large $x$, the first step in the proof is to invoke Lemma \ref{probru} (i) with $k=0$ and $K=\cK_0$, followed by an application of Lemma \ref{Ebo}. We get that \be \label{U11.2} N_{\cA}(x,x';y)=\int_{x'}^{x}\P_{z(t)}\(S_{z(t)}=0\)\,\dd t+ \O\(xe^{-4\l}\)\,+\O\(xe^{-\cK_0\log_2 x}\),\ee where the additional term $x'y^{2K_0+2}$ from Lemma \ref{probru} is absorbed into the first error term.
 By Lemma \ref{lo} and \eqref{decomp}, it is clear that \begin{align*} M_{\cA}(x;y) \ge\frac{N_{\cA}(x,x';y)-N_{\cA}(x,x';y')}{\log ^{1+c/2} x}.\end{align*}As a consequence, we have \alg{M_{\cA}(x;y)\ge \frac{N_{\cA}(x,x';y)}{\log^{1+c/2} x}+\O\bigg(x\exp\(-\(1-o(1)\)f\(1+\frac2{c}\)\log^{c/2} x\)\bigg).} Here, the error term is bounded using Lemma \ref{crux} with $\v_0=c$, $\v_1=c/2$ and $w=x'$. Following the same logic as in Section \ref{lower2}, evaluating the main term in \eqref{U11.2} reduces to bounding the probability $\P_{z_1}\(S_{z_1}=0\)$ from below. To this end, we define \[w\defeq \exp\(\frac1{\sqrt{2}}\sqrt{ \log_2 x\log_3 x}\),\quad \textrm{and}\quad u=(\log y)/\log w.\] By Lemma \ref{maierlem} ($z=w$), we find that there exists a choice of residues $\a=(\a_p\bmod{p}: p \le w)$ such that 
\be\label{l22}
S_{w}\le y\,\Theta_{w}\Big(1-\exp\big[-(1+o(1))u\log u\big]\Big),
\ee for sufficiently large $x$. As before, fix $\cF_w=\a$. We have \be\label{l2}\P_{z_1}\(S_{z_1}=0\)\ge \P_{z_1}\(S_{z_1}=0\mid \cF_w=\a\) 
\cdot \P_{w}\(\cF_w=\a\),\ee where $\P_w(\cF_w=\a)\ge e^{-2w}$ by a crude estimate using prime number theorem. Fix a choice of $\cF_{y^8}$ which extends $\cF_w=\a$. Applying \cite[Lemma 6.3]{ford2025}, it follows that 
\alg{
\P_{z_1}\(S_{z_1}=0\mid \cF_{y^8}\)&=
(1-\Theta_{y^8,z_1})^{S_{y^8}}(1+\O(y^{-5})).
}
For the conditional probability in \eqref{l2}, we may apply the law of total probability and sum over the possible values of $f$ of $\cF_{y^8}$ that extend $\cF_w=\a$. Indeed,
\alg{\P_{z_1}\(S_{z_1}=0\mid \cF_w=\a\)&=\sum_{f}\P_{y^8,z_1}\(S_{z_1}=0\mid \cF_{y^8}=f\)\,\P_{w,y^8}\(\cF_{y^8}=f\mid \cF_w=\a\) 
\\&\ge (1-\O(y^{-5}))\sum_{f}(1-\Theta_{y^8,z_1})^{S_{y^8}}\P_{w,y^8}\(\cF_{y^8}=f\mid \cF_w=\a\)
\\&=(1-\O(y^{-5}))\,\E_{w,y^8}\left[(1-\Theta_{y^8,z_1})^{S_{y^8}}\mid \cF_w=\a\right]
\\&\ge (1-\O(y^{-5}))\,(1-\Theta_{y^8,z_1})^{\E_{w,y^8}\left[S_{y^8}\mid \cF_w=\a\right]},
} where we use Lemma \ref{lem:tower} for the last inequality. We now handle the expectation in the same expression. We may calculate it as 
\alg{
\E_{w,y^8}\left[S_{y^8}\mid \cF_w =\a\right]&=\E_{w,y^8}\left[\sum_{n\in \ccS_w(y)}\mathds{1}_{\ccS_{y^8}}(n)\,\Big|\, \cF_w=\a\right]\\&=\sum_{n\in \ccS_w(y)}\P_{w,y^8}\(n\in \ccS_{y^8}\)=\Theta_{w,y^8}S_{w}\\&
\le y\,\Theta_{y^8}\Big(1-\exp\big[-(1+o(1))u\log u\big]\Big),
}where the last inequality follows from \eqref{l22}. Combining the last two estimates and substituting them into \eqref{l2}, we therefore conclude that 
\be\label{eq:central}
\begin{aligned}
\P_{z_1}\(S_{z_1}=0\)&\ge \exp\bigg[ 
    y\,\Theta_{y^{8}}\Big(1-\exp\big[-(1+o(1))u\log u\big]\Big)\cdot \Big(-\Theta_{y^{8},z_1}+\O\big(\Theta_{y^{8},z_1}^2\big)\Big)\bigg]\cdot e^{-2w}
    \\&=\exp\bigg[-\l\Big(1-\exp\big[-(1+o(1))u\log u\big]\Big)\Big(1+\O\big(1/\l^{21/20}\big)\Big)-2w\bigg].
\end{aligned}
\ee
We now treat the error terms in the above expression. We first note that 
\alg{
u=\frac{\log y}{\log w}=\(\sqrt{2}+o(1)\)\sqrt{\frac{\log_2 x}{\log_3 x}},
}
which further implies that \[u\log u=\(1/\sqrt{2}+o(1)\)\sqrt{\log_2 x\log_3 x}.\]
Therefore, by the lower bound on $\l$ and the sufficiently slow choice of $\eta(x)$ in \eqref{eq:lambda_thresholds},
\[
w=o\(\l\exp\(-(1+o(1))u\log u\)\).
\]
Combining these estimates together, we infer that
\[\P_{z_1}\(S_{z_1}=0\)\ge \exp\bigg[-\l\Big(1-\exp\big[-(1/\sqrt{2}+o(1))\sqrt{\log_2 x\log_3 x}\big]\Big)\bigg].\]
\smallskip
\noindent Substituting this lower bound into the estimate~(\ref{U11.2}), we conclude that \[N_{\cA}(x,x';y)\ge xe^{-\l}\exp\bigg[\l\exp\Big(-\big(1/\sqrt{2}+o(1)\big)\sqrt{\log_2 x\log_3 x}\Big)\bigg]\] for large $x$. This, in turn, implies that \[M_{\cA}(x;y)\ge N_{\cA}(x,x';y)/(2\log^{1+c/2} x)\ge xe^{-\l}\exp\bigg[\l\exp\Big(-\big(1/\sqrt{2}+o(1)\big)\sqrt{\log_2 x\log_3 x}\Big)\bigg],\] for sufficiently large $x$.\qed
\begin{rem}
    The constant $\sqrt{2}$ in Theorem \ref{thm:HLasymp} arises as a limitation of our techniques. To see this, write $s=\log w$. We must have $s<\log \l$. Assuming that the phase transition occurs around $\log \l=\O\(\sqrt{\log_2 x\log_3 x}\)$, we may see that
    \[u\log u\ge (1+o(1))\frac{\log_2x \log_3 x}{2s}.\] Now, for the error terms to be small in our probabilistic calculations, we need that
    \[\log \l \ge s+\frac{\log_2 x\log_3 x}{2s}+o\(\sqrt{\log_2 x\log_3 x}\).\] Since 
    \[s+\frac{\log_2 x\log_3 x}{2s}\ge \sqrt{2\log_2 x\log_3 x},\] our method cannot go below the constant $\sqrt{2}$. Equality occurs at \[s=\frac1{\sqrt{2}}\sqrt{\log_2 x\log_3 x},\] which is what we chose above.
\end{rem}

\section{Concluding remarks and further improvements}\label{6}

As noted in the introduction, our results locate the phase transition for these distributions on the scale 
\[\exp\(\Theta\(\sqrt{\log_2 x\log_3 x}\)\),\] although our current methods do not yet allow us to pinpoint the precise threshold within this scale. We anticipate that an improvement to the ranges in Theorems \ref{thm:HLgaps} and \ref{thm:HLpoisson} will require a more effective treatment of small primes. Currently, the probabilistic techniques utilized here and in \cite{ford2025} are effective only for somewhat larger primes, forcing us to apply the fundamental lemma up to a power of $\l$. Overcoming this limitation remains a major obstacle for future improvements.

\medskip
Conversely, strengthening the breakdown results in Theorems \ref{UK4} and \ref{thm:HLasymp} relies on refining the extremal interval sieve estimates. By appealing to known irregularities in the sieving process, our current bounds in \eqref{ir1} and \eqref{l2} exploit the fact that there exists a single, unusual choice of residue classes yielding a sieved set significantly smaller than its expected size. To push these limits further, it will be necessary to develop such methods capable of generating a large multiplicity of such extreme residue classes.
\bibliographystyle{amsplain}
\bibliography{main}

\end{document}